\documentclass[12pt]{amsart}

\usepackage{color,graphicx,amssymb,latexsym,amsfonts,txfonts,amsmath,amsthm}
\usepackage{wrapfig}
\usepackage{epsfig}
\usepackage{psfrag}
\usepackage{amssymb}
\usepackage{graphicx}
\usepackage{pb-diagram}
\newtheorem{teo}{Theorem}[section]
\newtheorem{cor}{Corollary}[section]
\newtheorem{lem}{Lemma}[section]

\newtheorem{rem}{Remark}[section]
\newtheorem{defin}{Definition}[section]

\def\Z{{\mathbb Z}}
\def\R{{\mathbb R}}
\def\C{{\mathbb C}}
\def\O{{\mathcal O}}
\def\G{{\mathcal G}}
\def\M{{\mathcal M}}

\def\L{{\mathcal L}}
\def\g{\gamma}
\def\S{\Sigma}

\def\menos{\hspace{-0.05cm}-\hspace{-0.05cm}}

\def\wt{\widetilde}

\def\b{\beta}
\def\a{\alpha}

\def\Im{{\rm Im}\,}

\begin{document}
\title[Limit points of the branch locus of $\M_g$. ]
{Limit points of the branch locus of $\M_g$.  }

\author{Raquel D\'{\i}az}
\address{Departamento de Geometr\'{\i}a y Topolog\'{\i}a. Facultad de Ciencias Matem\'aticas. Universidad Complutense de Madrid. Espa\~na}
\email{radiaz@mat.ucm.es}

\author{V\'ictor Gonz\'alez-Aguilera}
\address{Departamento de Matem\'atica. Universidad T\'ecnica Federico Santa Mar\'{\i}a. Valpara\'{\i}so, Chile}
\email{victor.gonzalez@usm.cl}
\thanks{
  \mbox{\hspace{11pt}}{\it Key words}: Moduli space, stratification, noded Riemann surfaces.\\
  \mbox{\hspace{11pt}}
The first author was partially supported by the Project MTM2012-31973. The second author was partially supported by Project   PIA, ACT 1415}


\begin{abstract}
Let $\mathcal{M}_{g}$ be the moduli space of compact connected hyperbolic  surfaces of genus $g\geq2$, and  ${\mathcal B}_g \subset {\mathcal M}_{g} $ its branch locus. Let  $\widehat{{\mathcal{M}}_{g}}$ be  the   Deligne-Mumford compactification of the moduli space of smooth, complete, connected surfaces of genus $g\geq 2$ over $\C$. The branch locus  ${\mathcal B}_g$ is stratified by smooth locally closed equisymmetric strata, where a stratum consists of  hyperbolic surfaces with equivalent action of their preserving orientation isometry group. Any stratum can be determined by a certain epimorphism $\Phi$. In this  paper, for any of these strata, we describe the topological type  of its limits points in $\widehat{\M}_g$ in terms of $\Phi$.   We apply our method to the $2$-complex dimensional stratum corresponding to the  pyramidal hyperbolic surfaces.
\end{abstract}
\maketitle

\maketitle
\section{Introduction}

Let $\mathcal{M}_{g}$ be the moduli space of Riemann surfaces of genus $g\geq2$, i.e., the space of complex structures on a compact connected topological surface $S$ up to isomorphism or, equivalently, the space of   hyperbolic surfaces of genus $g$ up to preserving orientation isometries. Still a third way of seeing $\mathcal{M}_{g}$ is as the moduli space of smooth, complete, connected curves of genus $g\geq 2$   defined over $\mathbb C$. The moduli space can be considered as the quotient of the Teichm\"uller space ${\mathcal T}_g$ by the action of  the modular group $\mbox{Mod}_g$  (or the mapping class group). The branch locus of the map ${\mathcal T}_g \rightarrow  {\mathcal T}_g/\mbox{Mod}_g \cong  {\mathcal M}_g$ is denoted by ${\mathcal B}_g$. Any point   $S \in {\mathcal B}_g$ has a nontrivial isotropy subgroup $G_S$ under the action of $\mbox{Mod}_g$ on  ${\mathcal T}_g$   and $G_S$ is isomorphic to the biholomorphic group of automorphisms of $S$. The points of  ${\mathcal B}_g$ can be organized in strata corresponding to surfaces with   isomorphic  non trivial automorphism groups   with the same topological action   on the surface. In this way the moduli space ${\mathcal M}_g$ is stratified  by Broughton into smooth locally closed strata \cite{Broughton}.
 
In \cite{abikoff}, Abikoff introduced the augmented Teichm\"uller space $\widehat{{\mathcal T}_g}$, by  adding
marked Riemann surfaces with nodes. The mapping class group acts on the augmented Teichm\"uller space, giving as  quotient the {\it augmented moduli space} $\widehat{\M_g}$. 
 
Earlier than that, Deligne-Mumford \cite{deligne} had compactified the moduli space by adding non-smooth stable curves. Their compactification is an  irreducible projective complex variety of dimension  $3g-3$, which contains $\mathcal{M}_{g}$ as a dense open subvariety. Harvey \cite{harvey} proved that the augmented moduli space and the Deligne-Mumford compactification are homeomorphic. More recently, Hubbard and Koch \cite{Hubbard} give an analytic structure to the augmented moduli space, making it isomorphic (in the analytic category) to the Deligne-Mumford compactification.

 The different topological types of the Riemann surfaces with nodes attached to $\M_g$ provide a stratification of  
 $\widehat{{\mathcal{M}}_{g}}$, the so called stratification by topological type. Each stratum is encoded combinatorially in terms of the weighted dual graph associated to the corresponding Riemann surface with nodes \cite{miranda}.

   

The purpose of this paper is the following.    
 Consider a  topological action of a group $G$ on a surface   $S$  of genus $g\geq 2$ by preserving orientation homeomorphisms. The action  induces  a branched covering  $p: S \rightarrow S/G\cong {\mathcal O}$ 
 and an epimorphism  $\Phi: {\pi}_1(\mathcal O,*) \rightarrow G$. The topological class of this action produces an equisymmetric stratum ${\mathcal L}=\M(g,\O,\Phi)$ of $\M_g$ (see Section \ref{sec:equisymmetric}). In this paper we describe the limit points of the stratum $\mathcal L$ in $\partial \M_g= \widehat{\M_g}\menos\M_g$, i.e., 
the intersection  $\widehat{\mathcal{L}}\cap \partial \M_g$, where $\widehat{\mathcal{L}}$ denotes the closure of $\mathcal{L}$ in $\widehat{\M_g}$. The main part is the combinatorial type of the strata in $\widehat{\mathcal{L}}$, and this is obtained in Theorem \ref{thm:Combinatotial}. Some topological properties of these strata are obtained in Theorem \ref{thm:Existence}. 
Finally, in order to illustrate how our results may be used, we work out  the case of the $2$-complex dimensional equisymmetric stratum of the pyramidal action of $D_n$ on surfaces of genus $n$.
Our methods for the proofs arise from a natural topological  and  hyperbolic point of view and are mainly based in the study of the $G$-invariant  admissible system of multicurves of the quotient orbifold $S/G$. 


Let us review some previous results related to the content of this paper. 
Let   $\mbox{Sing}\,{\mathcal M}_{g}\subset {\mathcal M}_g$ be  the sublocus of smooth, complete, connected curves of genus $g\geq 3$ over $\mathbb C$   with non trivial automorphism group. The irreducible components of the subvariety  $\mbox{Sing}\,{\mathcal M}_{g} $   are closely related to the stratification of ${\mathcal M}_g$ and have been characterized in \cite{cornalba}.  
The properties  of ${\mathcal B}_g$, mainly its topology and connectivity have been characterized in \cite{Broughton}, \cite{bartolini}, \cite{costa0}. 
The admissible degeneration of hyperelliptic  and trielliptic curves have been characterized in \cite{archer}. 
From a hyperbolic point of view, where a geodesic multicurve is pinched to length $0$, the topological monodromies of singular fibres of curves of genus three have been classified in \cite{ashikaga}.  In \cite{costa1}, it is  proved that the set of trigonal surfaces is connected in $\widehat{{\mathcal M}_g}$ and for $p\geq 11$ prime  the branch locus $B_{p-1}$ is disconnected in $\widehat{{\mathcal M}_{p-1}}$.  For any admissible group $G$, the nodal Riemann surfaces that are limits  of a $1$-dimensional equisymmetric stratum with the fixed signature $s=(0;[2,2,2,2,n])$, $n\geq3$ are characterized  in \cite{costa2}.

The content of the paper is organized as follows. In Section $2$ we give definitions, introduce notation and gather enough well known results together to provide the support of our statements and proofs.  In Sections $3$ and $4$ our main results are stated and proved. In Section $5$, as an application of the main theorem we work out  a complete description of the topological type of the limit points  of the $2$-dimensional equisymmetric  stratum corresponding to the  pyramidal Riemann surfaces \cite{victor}.

\section{Preliminaries}

\subsection{Orbifolds}
\label{sec:Orbifolds}
 In this section we will recall some notations and properties of the orbifolds that we will use in   the paper.  For the general theory about  orbifolds   see \cite{porti}.

We will denote an orbifold  by $\O$, and its underlying space by $|\O|$  (although, if there is no confusion, we will use $\O$ both for the orbifold and its underlying space).   The  local (or isotropy) group   of $\O$ at a point $x\in \O$ is denoted by $\Gamma_x$. The set of singular  points of $\O$ (points whose local groups are non-trivial) is denoted by ${\rm Sing}\,\O$.



Next, we will list   the type of orbifolds that will appear in the paper and some of theirs properties.

(I) Let  $S$ be a closed orientable surface  of genus $g$  and $H$ a finite subgroup of orientation preserving homeomorphisms. Then   the quotient $\O= S/H$ is a  closed (compact, without boundary) 2-orbifold   with a finite number of singular points, and the local groups $\Gamma_x$  at $x\in \O$   are cyclic groups of rotations. We call $x$ a {\it cone point} of order $m$ if  $\Gamma_x$ is  cyclic of order $m$.   The quotient map $p\colon S\to S/H$ is a regular branched covering.

The {\it (orbifold) fundamental group } $\pi_1(\O,*)$ of this kind of orbifolds is isomorphic to the quotient of $\pi_1(|\O|\menos{\rm Sing}\,\O)$ by  the 
  normal subgroup  generated by $\mu_x^{m_x}$ where $\mu_x$ is a loop in  $\pi_1(|\O|\menos{\rm Sing}\,\O)$ surrounding a disc which contains just the cone point $x$, and $m_x$ is the order of $x$. We remark that if $\a$ is a path in $\O$ from $*$ to a cone point  $P$  (and containing no other cone points), then   the loop $\a\a^{-1}$  is homotopic, in the orbifold fundamental group $\pi_1(\O,*)$,  to a path that goes along $\a$ until arriving to a small disc $D$ containing just the cone point  $P$, surrounds $D$, and comes back to $*$ along $\a^{-1}$. 

(II)  The unit interval  $I=\R/K$, where $K$ is the group generated by the two reflections $r_1(x)=-x, r_2(x)=-x+2$, is a 1-dimensional orbifold with two singular points, $\{0,1\}$.   
By definition, its fundamental group is  the group $K$, which is isomorphic to $\Z_2*\Z_2$.  It can also be interpreted as  the space of homotopy classes of loops $\pi_1(I,\frac{1}{2})$   generated by the loops $a,b$ where $a$ starts at $\frac{1}{2}$, goes until $0$ and comes back, and  $b$ starts at $\frac{1}{2}$, goes until $1$ and comes back, and subject to  the relations $a^2=b^2=1$.
\\
If $\O$ is a  2-orbifold, a simple  arc $\g$ in $\O$ joining two cone points of $\O$ of order 2 (and containing no other cone point)  inherits a structure of orbifold homeomorphic to $I$. We denote by $\g^a, \g^b$ the loops in $\g$ corresponding to the loops $a,b$ of $I$ by a homeomorphism.

(III) Let $\O$ be a 2-orbifold with ${\rm Sing }\,\O$  consisting only of  cone points. Let $\S=\{\g_1,\dots,\g_m\}$ be a collection of simple closed curves and simple arcs joining cone points of order 2 in  $\O$ and let $\O\menos\S$ denote $\O\menos\cup_{i=1}^m \g_i$. Then each connected component $\O_j$ of $\O\menos\S$ is an open (without boundary, non compact) suborbifold of $\O$ and its metric completion  $\O_j^c$  is an orbifold with boundary whose singular points are cone points. Here, for the metric completion, we are assuming any metric in $\O$ compatible with its topology. Actually, in the cases we will deal  with, the orbifolds  $\O_j$ will be   hyperbolic, and   we can consider the hyperbolic metric on $\O_j$, while the curves $\g_i$ can be considered to be geodesics.

  In what follows, a 2-orbifold will always be as in cases (I) or (III), that is,   orbifolds with orientable underlying space whose singular locus consists just of cone points.  The {\it signature } of $\O$ is $(\tau,c;m_1,\dots,m_k)$, where $\tau$ is the genus of $|\O|$, $c$ the number of boundary components of $\O^c$ and $\O$ has $k$ cone points of orders $m_1,\dots, m_k$. 
The Euler characteristic of an orbifold with signature   $(\tau,c;m_1,\dots,m_k)$ is defined to be 
  $$
  \chi(\O)=\chi(|\O|)-\sum_{i=1}^k(1-\frac{1}{m_i})=2-2\tau-c-\sum_{i=1}^k(1-\frac{1}{m_i}).
  $$ 
In particular, there is a finite list of  orbifolds with non-negative Euler characteristic.

  Notice that if $H$ is a group acting on an compact orientable surface $S$, giving as quotient an orbifold of signature $(\tau;m_1,\dots,m_k)$,  then the Riemann-Hurwitz formula can be expressed just as 
  $$\chi(S)=|G|\chi(\O).$$

A 2-orbifold admits a hyperbolic structure   if its interior is homeomorphic to the quotient of the hyperbolic plane by a Fuchsian group. 
It is well known that a closed  orbifold is hyperbolic  if and only if its Euler characteristic is negative. If the orbifold has boundary, the hyperbolic structure in its interior can be chosen to be of finite area.

 A homeomorphism of orbifolds is a   homeomorphism of the underlying spaces which takes cone points to cone points preserving their orders.

\begin{defin}
  Let $\O$ be a closed 2-orbifold of genus $g\geq 2$. An {\it (admissible) multicurve} in   $\O$ is a collection $\Sigma =\{\g_1,\dots, \g_k\}$ of disjoint   simple closed curves or simple arcs joining cone points of order 2 in the complement of the singular locus of $\O$   such that:
\begin{itemize}
\item[i)]  the $\g_i$ are homotopically different    in $\pi_1(\O\setminus Sing(\O))$;
\item[ii)]  none of them is homotopically trivial;
\item[iii)] none of them surrounds only one cone point;
\item[iv)] none of them bounds a disc with exactly two cone points of order 2.  
\end{itemize} 
\end{defin} 
We remark that the above conditions are equivalent to requiring that (the  metric completion of) each component of $\O\menos\S$  has negative Euler characteristic.  
  
The following lemma explains the reason why we only admit arcs joining cone points of order 2 (and not of other orders) in an admissible multicurve. 
  
 \begin{lem}\label{lem:multicurves}
 Let $S$ be a compact hyperbolic surface, $H$ a   subgroup of preserving orientation isometries  of $S$ with $\O=S/H$ the quotient orbifold and $p\colon S\to \O$ the associated branched covering.  
 Then, the preimage $p^{-1}(\S)$ of a multicurve $\S$ in $\O$ is a multicurve in $\S$. Conversely, if $\Gamma$ is a multicurve in $S$ invariant under $H$, then $p(\Gamma)$ is a multicurve in $\O$. 
 
 \end{lem}
 \begin{proof} 
 Let $\S$ be a multicurve in $\O$. If $ \g\in \S$ is a closed curve, then $p^{-1}(\g)$ is a union of closed curves. If $\g\in \S$ is an arc, because it joins cone points of order 2,  its preimage is also a union  of closed curves. Since the components of $\O\menos\S$ are hyperbolic,   their preimages under $p$ are also hyperbolic (because of the above formula relating the Euler characteristics of a component and its preimage). This shows that $ p^{-1}(\S)$ is an admissible  multicurve in $S$.
 
 Conversely, let $\Gamma$ be a multicurve on $S$ invariant under $H$. Because $S$ is a surface, all the components of $\Gamma$ are closed curves. An isometry   acting on a closed curve has either no fixed points or two fixed points. The quotient in the  first case is again a closed curve, while in the second case is an arc joining two cone points of order 2. 
 The remaining  conditions on the definition of multicurve are satisfied because each component of $S\menos\Gamma$  is hyperbolic, and so their images under $p$ are also hyperbolic.
 \end{proof}

\subsection{Stable hyperbolic surfaces}\label{sec:HypSurfWithNodes}

 \subsubsection{Stable hyperbolic surfaces.}
We will expressed most of our statements and results in terms of hyperbolic surfaces, rather than in terms of Riemann surfaces. The 
 reader will have no difficulty in making the appropriate translation. Let $S$ be a  closed orientable  surface  of genus
$g$, and let ${\mathcal F} \subset S$ be an admissible multicurve, i.e., a collection
(maybe empty) of homotopically independent  pairwise disjoint
simple loops. Consider the quotient space $S  =S/\mathcal{F}$ obtained by
identifying the points belonging to the same curve in ${\mathcal
F}$. We say that $S$ is a {\it stable surface of genus $g$} and
that each element of $S$ which is the projection of a curve in
${\mathcal F}$ is a {\it node} of $S$. We denote by $N(S) \subset
S$ the collection of nodes of $S$ and we say that each connected
component of $S\hspace{-0.05 cm}\menos N(S)$ is a {\it part} of $S$. In particular, a
stable surface $S$ with $N(S)=\emptyset$ is just a closed surface.

 As was seen in Section \ref{sec:Orbifolds}, each part admits a hyperbolic structure, so it may sense to define
a  {\it stable hyperbolic surface} of genus $g$    as  a stable surface $X$ of genus $g$ so that each of its parts  has a complete hyperbolic structure  of finite area (that is, each part is a hyperbolic  surface with punctures).   An {\it isometry} $h$ between two stable hyperbolic surfaces is a homeomorphism $h$ whose  restriction is an isometry  on the complement of the nodes.

\subsubsection{Correspondence with stable graphs.}
Next, we will associate to each stable hyperbolic surface (or a Riemann surface with nodes) a combinatorial object.

 A {\it stable graph} is a connected  weighted graph $\mathfrak G$, where  each vertex with weight zero
has  degree at least three (the {\it degree} of a vertex is the number of edges coming into it, taking into account that loops count by two). 
The genus of $\mathfrak G$ is defined as 
$$g=\displaystyle\sum_{i=1}^{v_{\mathfrak G}} g_i+ e_{\mathfrak G}-v_{\mathfrak G}+1,
$$
where $v_{\mathfrak G}$ is the number of vertices and $e_{\mathfrak G}$ is the number of edges and the $g_i$ is the weight of the vertex $v_i$ of $\mathfrak{G}$. An isomorphism between stable graphs is a usual graph isomorphism preserving  the weights. Note that in the previous formula, a loop  counts as one edge.

To a stable hyperbolic surface   $X$  we can associate  the stable graph ${\mathcal G}(X)=(V_{X},E_{X},w)$, where $V_{X}$ is the set of vertices, $E_{X}$ is the set of edges,
and $w$ is a function on the set $V_{X}$ with non-negative integer values.
This triple is defined in the following way:

\begin{enumerate}
\item To each   part of $X$ corresponds a vertex in $V_{X}$.

\item To each node in the boundary of  two parts  
corresponds and edge in $E_{X}$ connecting the corresponding vertices.
Multiple edges between the same pair of vertices and loops are allowed in
${\mathcal G}(X)$.

\item The function $w:V(X)\rightarrow{\mathbb{Z}}_{\geq0}$ associates to each
vertex the genus $g_{i}$ of the corresponding part.
\end{enumerate}
Notice that  if a part of $X$ has genus $0$, since it is hyperbolic,  it should  have at least three
punctures; therefore, the associated graph is stable.

Conversely, given a stable graph, we can easily find a stable hyperbolic surface $S$    such that $\mathfrak G \cong {\mathcal G}(S)$.

Two hyperbolic surfaces with nodes $X_{1}$ and $X_{2}$ are homeomorphic if and only if
their stable graphs ${\mathcal G}(X_{1})$ and ${\mathcal G}(X_{2})$ are isomorphic as weighted graphs \cite{miranda}. Therefore,   stable graphs up to isomorphism are in bijection with  homeomorphism classes of   hyperbolic stable surfaces.

We end with the following remark. For each hyperbolic  surface with nodes $X$ its preserving orientation isometry group    ${\rm Iso}^+(X))$ is finite.  Also, the group of automorphisms of the graph $\G(X)$, which we denote by   $\mbox{Aut}({\mathcal G}(X))$, is finite.  Any $\varphi \in {\rm Iso}^+(X)$ induce an
automorphism  of ${\mathcal G}(X)$, therefore   there is a  homomorphism $\theta_X:{\rm Iso}^+(X) \rightarrow \mbox{Aut}({\mathcal G}(X))$. The homomorphism $\theta_X$ is, in general, neither injective nor surjective.

\subsubsection{Augmented moduli space} The augmented moduli space of genus $g$ is the space $\widehat{\M}_g$ of  stable hyperbolic surfaces of genus $g$ up to isometry. This space provides a compactification of moduli space, once given a topology that intuitively works as follows. Consider a point  $X\in \widehat{\M}_g$ with nodes $N_1,\dots, N_r$.  A sequence $X_n\in \M_g$ converges to   $X\in \widehat{\M}_g$  if there is a family of geodesics $\mathcal{F}_n=\{\a_1^n,\dots, \a_r^n\}$ in $X_n$ so that:
\begin{itemize}
\item[(i)] the stable surface $X_n/\mathcal{F}_n$ is homeomorphic to $X$;
\item[(ii)]    the length of each $\a_i^n, i=1,\dots, r,\,$ tends to zero when $n\to \infty$; and 
\item[(iii)]  away from the curves $\a_i$ and the nodes, the hyperbolic surfaces $X_n$ are close to $X$ when $n\to \infty$.
 \end{itemize}
A way of formalizing this is by first considering marked surfaces and the augmented  Teichm\"uller space. See for instance \cite{Hubbard} for details.

 The augmented moduli space is stratified in the following way. For each stable graph  $\mathfrak G$, we consider the space ${\mathcal E}({\mathfrak G})\subset  \widehat{\M}_g$   of stable hyperbolic surfaces $X$ whose associated graph $\G_X$ is isomorphic to ${\mathfrak G}$. This space is homeomorphic to  the product of the moduli spaces of the vertices of $\G_X$, i.e., moduli spaces of punctured surfaces. (See \cite{Hubbard} for a detailed exposition of this.)
Notice that the codimension of the stratum ${\mathcal E}({\mathfrak G})$ in  
$\widehat{\M}_g$ is equal to the number of edges of ${\mathfrak G}$, or the number of nodes of the stable hyperbolic surface $X$.
Since there is a finite number of non-isomorphic  stable graphs of a fixed genus $g$, then $\widehat{   {\mathcal M}_g}$ decomposes into a finite number of strata ${\mathcal E}({\mathfrak G})$, one for each stable graph. This is called the {\it topological stratification} of $\widehat{\M}_g$.

The problem of determining  the number of non isomorphic stable graphs of a fixed genus $g\geq2$  (or the number of different strata of $\widehat{\M}_g$), is not an easy problem which is not in the scope of this paper. Nevertheless, as a curiosity for $g=2$ there are seven different stratum including the smooth curves; two $0$-dimensional stratum, two $1$-dimensional stratum and two $2$-dimensional strata. For $g=3$, it can be checked that there are forty two different strata.

 \subsection{Equisymmetric stratification.}
 \label{sec:equisymmetric}
 
 \subsubsection{Topological actions.}
 Let $G$ be a finite group acting on a surface $S$ by orientation-preserving homeomorphisms, i.e., there is a 
monomorphism $\iota\colon G\to Homeo^+(S)$. 
Two actions $(S_1,G,\iota_1)$, $(S_2,G,\iota_2)$ of $G$ on $S_i,i=1,2$ 
are {\it topologically equivalent} if there is a  homeomorphism $f\colon S_1\to S_2$ and an isomorphism $\varphi$ of $G$ such that for each $g\in G$ the following diagram commutes
$$
\label{diagram}
\begin{diagram} \node{S_1}\arrow{e,t}{\iota_1(g)}\arrow{s,l}{f}  \node{S_1}\arrow{s,r}{f}\\ \node{S_2}\arrow{e,b}{\iota_2(\varphi(g))}  \node{S_2} \end{diagram} 
$$

The quotient space $\O=S/\iota(G)$ is an orbifold, and the quotient map
  $p\colon S\to \O$ is a regular branched covering whose automorphism group is $\iota (G)$. By covering theory, there is an epimorphism $\bar\Phi\colon \pi_1(\O,*) \to \iota(G)$ with kernel $p_*(\pi_1(S,\tilde *))$, where $*$ is a basepoint in $\O$ and $\tilde *$ is a point in $p^{-1}(*)$.  We denote  $\Phi=\iota^{-1}\circ \bar\Phi$.
 
Conversely, any epimorphism $\Phi\colon\pi_1(\O,*) \to G$   with kernel isomorphic to the fundamental group of a surface $S$ determines an action of $G$ on $S$ up to 
  topological equivalence. In this setting,  the topological  equivalence of actions can be restated as saying that two epimorphisms  $\Phi_i\colon \pi_1(\O_i,*_i)\to G$ are {\it equivalent} if there is a homeomorphism $h\colon \O_1\to \O_2, h(*_1)=*_2$ and an isomorphism   $\varphi$ of $G$ such that the following diagram commutes:
  \[
\begin{diagram} \node{\pi_1(\O_1,*_1)}\arrow{e,t}{\Phi_1}\arrow{s,l}{h_*}  \node{G}\arrow{s,r}{\varphi}\\ \node{\pi_1(\O_2,*_2)}\arrow{e,b}{\Phi_2}  \node{G} \end{diagram}
\]
  (notice that, when proving the second diagram from the first one,  $h$ is  the homeomorphism  induced by the homeomorphism $f$  and the isomorphism $\varphi$ is the same). 
 
 We remark that the classification of actions of finite groups on surfaces is a problem of great complexity if the genus of the surface is big. For   genera  2,3,4, the complete classification can be seen in 
  \cite{Broughton-Classifying} and \cite{kimura}.

 \subsubsection{Equisymmetric locus.}
 
 Let $G$ be a finite group, $\O$ be an orbifold and  $\Phi\colon \pi_1(\O,*)\to G$  be an epimorphism with kernel isomorphic to the fundamental group of a surface $S$. Consider the topological class of action   determined by   $\Phi $.  We define the {\it equisymmetric stratum} of $\mathcal{M}_g$ determined by $\Phi$ as the set
$\mathcal{M}_g(G,\O,\Phi)$ consisting of the hyperbolic surfaces $X\in \M_g$ such that the action  
$(X,{\rm Iso}^+(X),i)$ (where $i$ is the inclusion)  is topologically equivalent to the action given by $\Phi $.

When considering all possible actions of finite groups on $S$, these loci determine a stratification of the moduli space $\mathcal{M}_g$ known as the {\it equisymmetric stratification of moduli space}, see \cite{Broughton}.

 \section{Strata in the boundary of equisimmetric loci}  \label{sec:strata}

 In this section (and on the remaining of the paper) we will fix an action of a finite group $G$ on a surface $S$ given by an epimorphism $\Phi\colon \pi_1(\O,*)\to G$. For brevity, 
 let us denote by $\mathcal{L}$ the equisymmetric stratum $ \mathcal{M}_g(G,\O,\Phi)$.  Let $\partial \mathcal{L}$ be  the set of points $Z\in\widehat{\M}_g\menos\M_g$ which are in the closure of $\mathcal{L}$ as a subset of $\widehat{\M}_g$.
 
 The intersection of $\partial\mathcal{L} $ with the strata of $\widehat\M_g$ provides a stratification of $\partial\mathcal{L} $. 
The next theorem gives a first description of the strata of $\partial\mathcal{L} $. We first need a notion of equivalent admissible  multicurves on an orbifold.


\begin{defin}
Let $p\colon S\to \O$ be a regular  branched covering determined by an epimorphism $\Phi\colon \pi_1(\O,*)\to G$.  Two (admissible) multicurves $\S,\S'$ in   $\O$ are said to be {\it  equivalent under $\Phi$}, written as $\S\sim_{\Phi}\S'$,   if there is a homeomorphism $h_* $ of $\O$ fixing $*$,  and there is an isomorphism $\varphi$ of $G$ such that $h(\S)=\S'$  and $\varphi\circ\Phi\circ h_*^{-1}=\Phi$. 
\end{defin}


  \begin{teo}\label{thm:Existence} 
  Let $\mathcal{L}=\mathcal{M}_g(G,\O,\Phi)$.
Then we have:  
  \begin{itemize}
  \item[(a)] Each  stratum of  $\partial\mathcal{L} $ is connected. 
   \item[(b)]  
   There is a surjection $\Psi$ from the set of equivalence classes of multicurves in $\O$ onto the set strata of $\partial\mathcal{L} $. 

  \item[(c)] If the signature of $\O$ is $(\tau; m_1,\dots, m_r)$ and $\S=\{\g_1,\dots, \g_k\}$ is a multicurve in $\O$, then  the dimension of the stratum $\Psi(\S)$ is    equal to $3\tau-3+r-k$.
 %
 %
  \end{itemize}

 \end{teo}
\begin{proof}

(a)   Consider the stratum $\mathcal{E}(\mathfrak{G})$ of $\partial\M_g$ corresponding to the stable graph $\mathfrak{G}$. We need to show that $\mathcal{E}(\mathfrak{G})\cap\partial\mathcal{L} $ is connected. We have
$$
\mathcal{E}(\mathfrak{G})=\prod_{V\in V(\G)}\M_{w(V),d(V)},
$$
where $w(V)$ is the weight of the vertex $V$, $d(V)$ its degree, and $\M_{g,b}$ is the moduli space of surfaces of genus $g$ and $b$ punctures. If $ \mathcal{E}(\mathfrak{G})\cap \partial \mathcal{L}\not=\emptyset$, then the group $G$ acts on the graph $\mathfrak{G} $. Let 
$\mathcal{W}_1,\dots,\mathcal{W}_s$ be the orbits  with more than one vertex  under this action. 
Then, $\mathcal{E}(\mathfrak{G})\cap \partial \mathcal{L} $ is homeomorphic to the product of the  moduli spaces of one vertex in each of the orbits $\mathcal{W}_i$ and certain  equisymmetric loci in the  moduli spaces of the vertices fixed by $G$. Since equisymmetric loci are connected (see \cite{Broughton}),  we have the result. 
  

(b) We first define the map $\Psi$. Let $\S=\{\g_1,\dots, \g_k\}$ be a multicurve in $\O$. Since $p^{-1}(\S)$  is  a multicurve 
in $S$ (by Lemma \ref{lem:multicurves}), we can consider the stratum $\mathcal{E}(p^{-1}(\S))$ of $\widehat\M_g$ to which  the stable surface $S/p^{-1}(\S)$ belongs.
 We define $\Psi([\S])= \mathcal{E}({p^{-1}(\S)})\cap \partial \mathcal{L}$. 
 We need to check that this intersection is non-empty and that the definition is independent of the  chosen representative of the equivalence class $[\S]$. 

First, since all the components of $\O\setminus\Sigma$ have negative Euler characteristic (by the definition of multicurve), then there is a hyperbolic structure on $\O$ with all the curves in $\Sigma$ pinched, 
producing a hyperbolic orbifold  with nodes. This hyperbolic structure lifts to a 
 hyperbolic structure on  
$S/p^{-1}(\Sigma)$, i.e., we have a stable curve $X$ in 
$\mathcal{E}({p^{-1}(\S)})\cap \partial \mathcal{L}$, and so this set is non-empty. 
 
To show  that   $\Psi$ is well defined  is just an observation, due to the fact that, if   $\S\sim_{\Phi}\S'$, then the homeomorphism $h$ determines a covering homeomorphism $ f$ 
with    $ f(p^{-1}(\S))=p^{-1}(\S')$. Thus the stable surfaces  $S/p^{-1}(\S)$ and  $S/p^{-1}(\S')$ are isomorphic, so they belong to the same stratum.

Finally, let us see that $\Psi$ is onto. Let $\mathcal{E}(\mathfrak{G})\cap\partial \mathcal{L}$ a non-empty stratum. Take $X\in \mathcal{E}(\mathfrak{G})\cap\partial \mathcal{L}$  and $X_n\in \L$ a sequence converging to $X$. For $n$ sufficiently large and $\varepsilon$ sufficiently small, the set of curves in $X_n$ with length less than $\varepsilon$ is a multicurve $\Gamma$ invariant under $G$ (these are the curves which converge to the nodes of $X$).  Let $\S=\Gamma/G$. Then, the construction shows that  $\Psi([\S])=\mathcal{E}(\mathfrak{G})\cap\partial \mathcal{L}$.


 (c) The stratum $\Psi(\S)$, with  $\S=\{\g_1,\dots, \g_k\}$, can also be seen as the moduli space of the orbifold $\O$ with $k$ curves or arcs pinched. Since each curve or arc pinched reduces the dimension by 1, we have the result. 
  \end{proof}

\begin{rem}   In the situation of Theorem \ref{thm:Existence}, the equisymmetric locus $\mathcal{L}$ can be determined from the moduli space of the orbifolds $\O$ with signature  $(\tau; m_1,\dots, m_r)$. That is, if $X$ is a hyperbolic orbifold homeomorphic to $\O$ and $H\colon \pi_1(\O,*)\to {\rm Iso}^+ \mathbb H^2$  is its holonomy representation, then the holonomy representation of $p^{-1}(X)$, which is an element of $\mathcal{L}$, is the restriction of $H$ to the kernel of $\Phi$.  Conversely, the holonomy representation of any point in   $\mathcal{L}$ can be obtained in this way. In similar way, each stratum of $\partial \mathcal{L}$ is determined by a stratum in the augmented moduli space of $\O$.  
\end{rem}

\section{Stable graph determined by a multicurve.}\label{sec:covering}

\subsection{Notations about the action.}
As in the previous section, we consider a fixed action $(S,G,\iota)$, where $S$ is a surface of genus $g$ and $G$ a finite group. Let  $p\colon S\to \O$ be the associated branched covering, $\Phi\colon \pi_1(\O,*)\to G$ the associated epimorphism,    and let $\mathcal{L}=\M_g(G,\O,\Phi)$ be the corresponding equisymmetric locus. 
 In  Theorem \ref{thm:Existence}, we defined a map $\Psi$  which assigns to any multicurve $\S$ in $\O$ a stratum of $\partial\mathcal{L}$, namely, the stratum corresponding to the stable surface $S/p^{-1}(\S)$. The stable graph corresponding to $S/p^{-1}(\S)$ is denoted by  $\G_{\S}$ and called {\it stable graph determined by $\S$ under $p$}. In Theorem \ref{thm:Combinatotial}, we will describe  $\G_{\S}$
  in terms of $\Phi$.

Before stating the theorem, we need to give some notations and preliminary results.

Paths are always maps from the unit interval $[0,1]$ to a space, and we will use the same notation for  both  the map and  its image. 
Also by abuse of notation, usually  we will not distinguish between a loop   and its homotopy class. 
 We will also use the notations $g\cdot x= \iota g(x)=\iota(g)(x)$ and $g\cdot Z=\iota(g)(Z)$, for $g\in G$ and $x\in S$, $Z\subset S$.

Fix basepoints $*\in\O\setminus{\rm Sing}\,\O$ and $\tilde *\in p^{-1}(*)$. We recall the definition of $\Phi=\iota^{-1}\circ \bar\Phi$ from covering theory. 
The epimorphism $\bar\Phi\colon \pi_1(\O,*) \to \iota(G)$ is defined by the condition $\bar\Phi([\a])(\tilde{*}) =\tilde{\a}(1)$, where   $[\a]\in\pi_1(\O,*)$ and $\tilde{\a}$ is the unique  lift of $\a$ starting at $\tilde{*}$. Clearly, ${\rm ker}\,\bar\Phi=p(\pi_1(S,\tilde{*}))$.
   We also recall that if a covering homeomorphism $h$ maps $x$ to $x'$, and if   $\tilde \b , \widetilde \b'$ are, respectively, the unique lifts starting at $x,x'$ of a same path $\b$  then
 $h(\tilde \b(1))=\widetilde \b'(1)$.

\subsection{Homomorphisms $\Phi_X$ induced by $\Phi$.}\label{sec:InducedHomo} 
%
%
Let  $X\subset \O$ be: either a 2-dimensional suborbifold of $\O$,   or a closed curve in $\O$ not containing singular points, or a simple arc whose endpoints are cone points of $\O$ of order 2 and no containing other cone points. 
We choose a basepoint $*_X\in X\setminus {\rm Sing}\,X$  and a path $\b_X$ from $*$ to $*_X$ and we  consider the homomorphism  $i_*\colon \pi_1(X,*_X)\to \pi_1(\O,*)$ defined by $i_*(\a)=\b_X\a\b_X^{-1}$. Now, define $\Phi_X=\Phi\circ i_*$. Notice that $i_*$   depends on the choice of $\b_X$, but only weakly: if we choose another $\b'_X$, then the analogous homomorphism $i'_*$  is conjugate to $i_*$.

 \medskip

Next lemma will   be useful for Theorem \ref{thm:Combinatotial}. It  concerns the stabilizers of connected components of $p^{-1}(X)$, for $X\subset \O$. 
 If $Z\subset S$, its {\it stabilizer} is the subgroup   
$
{\rm Stab}_Z =  \{g\in G\,\colon \,   g\cdot Z=Z 
 \}$.


 \begin{lem}  
 \label{lem:stabilizers}
 Let $S, G, \O , p, *, \tilde *$ and $\Phi$  be as above, and let $X $  be a connected subset of $\O$. Choose $*_X, \b_X$ as above. Let $\widetilde \b_X$ be the unique lift of $\b_X$ starting at $\tilde *$, and   denote  $\tilde *_X=\wt\b_X(1)$. Finally,  let  $\mathcal{C}$ be the connected component of $p^{-1}(X)$ containing $\tilde *_X$. Then
\begin{itemize}
\item[(a)]  ${\rm Stab}_{\mathcal C }={\rm Im}\Phi_X$, and $\mathcal C$ contains the points $h\cdot \tilde*_X$ for $h\in {\rm Im}\Phi_X$.
\item[(b)]  If $\mathcal C'$ is another component of $p^{-1}(X)$, then there exists $g\in G$ so that $\mathcal{C}'=g\cdot \mathcal{C}$,   ${\rm Stab}_{\mathcal C' }=g{\rm Im}\Phi_Xg^{-1}$, and $\mathcal C'$ contains the points $g {\rm Im}\Phi_X\cdot \tilde*_X$.
\item[(c)]
 The number of components of $p^{-1}(X)$ is $|G|/|{\rm Im}\Phi_X|$.\end{itemize} 
\end{lem}
  
 \begin{proof}
(a) 
  Let $\a\in\pi_1(X,*_X)$ and $\tilde \a$ its lift from 
 $  \tilde *_X $. Since $\tilde *_X\in\mathcal C$, then   $\tilde \a \subset \mathcal{C}$.  By the definitions of $\Phi,\tilde \a$ and $\tilde \b$,   the covering automorphism $\Phi_X(\a)=\Phi(\b_X\a\b_X^{-1})$ maps  $\tilde *$ to $\widetilde{\b_X\a\b_X^{-1}}(1)$, and the trace of the path  $\widetilde{\b_X\a\b_X^{-1}}$ contains the traces of the paths $\tilde{\b}$ and $\tilde{\a}$. Now the   lift of $\b$ starting at $\tilde *$ is $\tilde\b$ and ends at $\tilde *_{\g}$, while the   lift  of $\b$ starting at $\widetilde{\b_X\a\b_X^{-1}}(1)$ ends at $\tilde\a(1)$. Hence, $\Phi_X(\a)$ maps $\tilde *_{X}$ to  $\tilde\a(1)$, and so it preserves $\mathcal C$.
 
 On the other hand, let $h\in {\rm Stab}_\mathcal C$ and let $\tilde y=h(\tilde *_X)$. Let  $\tilde w$ be a path in $\mathcal C$ from $\tilde *_X$ to $\tilde y$. Then $p\circ \tilde w$ is  a loop in $X$ based on $*_X$ and, by the previous paragraph, $\Phi_X(p\circ \tilde w)( \widetilde*_X)  =\tilde y $. Hence $h=\Phi_X(p\circ \tilde w)$, i.e., $h\in {\rm Im}\,\Phi_X$.
  
Part (b) is clear. 

(c) A point $x\in X$ has $|G|$ preimages. By (a) and (b), each component of $p^{-1}(X)$ contains $|{\rm Im}\Phi_X|$ preimages of $x$. Thus, 
the result follows. 
  \end{proof}


\subsection{The multicurve $\S$.}\label{sec:Multicurve} 
 Let $\O$ be an orbifold and $\S=\{\g_1,\dots, \g_k\}$ be a multicurve in $\O$. The  complement $\O\setminus \cup_i \g_i$ is a union of (open) suborbifolds   $\O_1,\dots, \O_r$. We denote by $\bar \O_j$ the closure of $\O_j$ in $\O$. We do the following definitions
$$
\S_j=\{\g\in \S  \,:\, \g\subset\bar \O_j \}\quad \hbox{for each } j=1,\dots, r
$$ 
 $$
 \S^1=\{\g\in \S \,:\, \hbox{there exists a unique } j\in\{1,\dots,r\} \hbox{ with } \g\subset \bar\O_j \}
 $$
  $$
 \S^2=\{\g\in \S \,:\, \hbox{there exist different } j,j'\in\{1,\dots,r\} \hbox{ with } \g\subset \bar\O_j\cap \bar\O_{j'} \}
 $$

 It will be convenient to  collect the combinatorial  information given by the suborbifolds $\O_j$ and the curves in $\S^2$ in  a graph  $\mathcal{R}$, as follows. The vertices of $\mathcal{R}$ are the suborbifolds $\O_1,\dots, \O_r$ and  there is an edge between two different vertices $\O_j,\O_{j'}$ if and only if there is a $\g\subset \bar O_j\cap \bar\O_{j'}$ (this graph does not have loops). We choose a    spanning tree $\mathcal{T}$ of $\mathcal R$.

\subsection{Choice of basepoints and paths $\b_{j,\g}, \b_j,\b_{\g}$.} 
\label{sec:ChoicePaths} We will choose basepoints and paths joining them in a convenient way that we explain next (see Figure \ref{fig:CaminosBeta}). 

\begin{itemize}
\item[(1)] For each $j=1,\dots , r$, take a basepoint $*_j$ in $\O_j\setminus {\rm Sing}\,\O$. For any  $\g=\g_i\in \S$ choose  a basepoint  $*_{\g}\in\g\setminus{\rm Sing}\,\O$.

\item[(2)] Let$j=1,\dots ,r$.   For  each $\g\in\S^2\cap \S_j$ we  consider a  simple    path 
$\b_{j,\g}$ from $*_j$ to $*_{\g}$.
For 
   each $\g\in\S^1\cap \S_j$  we consider two   simple paths  $\b_{j,\g}^a$, $\b_{j,\g}^b$
 from $*_j$ to $*_{\g}$ such that $\b_{j,\g}^a(\b_{j,\g}^b)^{-1}$ intersects   $\g$ exactly once and, in the case that $\g$ is an arc, $\b_{j,\g}^a(\b_{j,\g}^b)^{-1}$ bounds a disc which contains just one of the endpoints of $\g$  and no other cone point (thus, $\b_{j,\g}^a(\b_{j,\g}^b)^{-1}$ is freely homotopic to either  $\g^a$ or $\g^b$).  Moreover we choose all these paths   so that  they are disjoint except at their endpoints.
\item[(3)] We choose one of the $*_j$ as basepoint for $\O$, for instance $*=*_1$. We choose  $\b_1$ to be the constant path, and we will choose the paths  $\b_j,\b_{\g} (j=2,\dots, r, \g\in \S)$ going along the paths   that we have already chosen  in (2), as follows.  For $j=2,\dots, r$, let $T_j=\O_1\g_{i_1}\O_{j_2}\g_{i_2}\dots \O_j$ be the path in the tree $\mathcal T$ from $\O_1$ to $\O_j$, given by its sequence of vertices and edges (notice that all the curves $\g$ in $T_j$ are in $\S^2$). The path $\b_j$ is determined from $T_j$, replacing each occurrence of $\O_{k}\g \O_{k'}$ by $\b_{ k,\g }\b_{k',\g }^{-1}$. 
Finally, let $\g\in\S$. If $\g\in \S^1$, then we take $\b_{\g}=\b_j\b_{j,\g}^a$. Otherwise, if $\g\in\S_j\cap \S_{j'}$ for $j\not=j'$, we choose one of the indexes $j,j'$, for instance $j$, and take $\b_{\g}=\b_j\b_{j,\g}$.

\end{itemize}

 \psfrag{*}{$*$}

\psfrag{*j1}{$*_{j_1}$}
\psfrag{*j2}{$*_{j_2}$}

\psfrag{*g}{$*_{\gamma}$}
\psfrag{*j}{$*_{j}$}
\psfrag{bj1}{$\beta_{j_1}$}
\psfrag{bj2}{$\beta_{j_2}$}
\psfrag{bg}{$\beta_{\gamma}$}
\psfrag{bj1g}{$\beta_{j_1,\gamma}$}
\psfrag{bj}{$\beta_{j}$}
\psfrag{bjga}{$\beta_{j,\gamma}^a$}
\psfrag{bjgb}{$\beta_{j,\gamma}^b$}

\psfrag{bj2g}{$\beta_{j_2,\gamma}$}
\psfrag{D}{$\Delta$}
\psfrag{g}{{\bf $ \gamma$} }
\psfrag{d}{$ {\delta}$}

\psfrag{Oj1}{$\bf{\mathcal{O}_{j_1}}$}
\psfrag{Oj2}{$\bf{\mathcal{O}_{j_2}}$}
 \psfrag{D}{$\Delta$}
\psfrag{g}{{\bf $ \gamma$} }
\psfrag{d}{$ {\delta}$}

\psfrag{Oj}{$\bf{\mathcal{O}_{j}}$}

\begin{figure}
\center
\includegraphics[height=6cm,width=15cm]{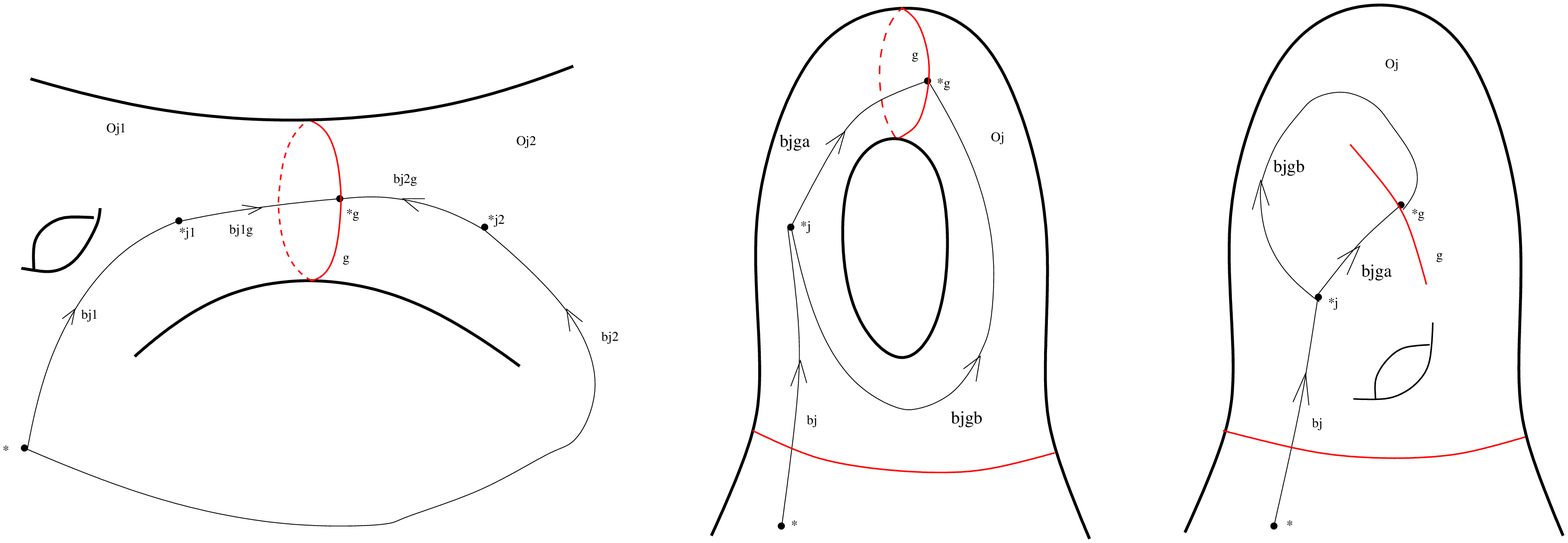}
\caption{}
\label{fig:CaminosBeta}
\end{figure}

\subsection{The polygon $P$.}\label{sec:PolygonP} 
By a {\it closed polygon} we mean a space homeomorphic to a closed Euclidean polygon with a finite number of sides. In particular, we can speak of {\it  sides} (that we will consider  open segments) and {\it  vertices} of the polygon. 
A {\it polygon} will be a closed polygon minus  some subset of its vertices and edges. The {\it boundary } of the polygon is the union of its sides and vertices.
 
In next lemma we will cut the orbifold  $\O$ into  a   polygon with some additional requirements. This will be used to label the points in the covering space $S$ in terms of  $\O$ and $\Phi$. In the lemma, the arcs $\b_j,\b_{\g},\b_{j,\g}$ are thought both as   maps from the unit interval to $\O$ and as the images if these maps.

\begin{lem}\label{lem:PolygonP}
There is   polygon $P$  and a continuous bijection $b\colon P\to \O$ so that 
\begin{itemize}
\item[(a)] $ b^{-1}({\rm Sing}\, \O) $ is contained in the boundary of $P$;  
\item[(b)] for each $j=1,\dots ,r$ and for each $\g\in\S$, the maps $b^{-1}\b_j$ and $b^{-1}\b_{\g}$ are continuous.
\item[(c)] for each $j $ and each $\g\in\S_j\cap \S^2$, the restriction of the map  $b^{-1}\b_{j,\g}  $  to the interval $[0,1)$ is continuous;
\item[(d)] for each $j$ and  $\g \in\S_j\cap \S^1$, the restriction of the maps  $b^{-1} \b_{j,\g}^a, b^{-1} \b_{j,\g}^b   $  to $[0,1)$ are continuous.
\end{itemize}
\end{lem}
\begin{proof}
We will obtain $P$  in several steps, by cutting the orbifold $\O$ along a set $W$  of curves. 
At each step, when we say ``cut $R$ along $Z\subset R$" we mean ``remove $Z$ from $R$ and then do the metric completion". In this way, we obtain a pair $(Q,b)$ where  $Q$ is the  space  $(R\menos Z)^c$, and   $b$ is a surjective map $b\colon Q\to R$ which is the identity in $R\hspace{-0.05cm}-\hspace{-0.05cm}Z$ and extends by continuity to the remaining points of $Q$.

(I) Cut $\O$ along all the curves $\g\in \S^1$ (if any).  We obtain $(Q_1,b_1)$, with    $Q_1$ an orbifold with boundary (maybe empty). Notice that, if $\g\in\S$ is an arc contained in $\bar \O_j$ then   the curve $ (\b_{j,\g}^b)^{-1}\b_{j,\g}^a$ separates $Q_1$ into two components, one of them is topologically a disc, which we denote by $D_{\g}$.

(II) Cut $Q_1$ along 
all the curves $\g$ in $\mathcal R\menos \mathcal T$, where $\mathcal T$ is the spanning tree in the graph $\mathcal R$ defined in Section \ref{sec:Multicurve}. We obtain $(Q_2,b_2)$, with $Q_2$ another orbifold with boundary. 

Consider the set 
 $$ 
\mathcal B=\bigcup_{\substack{j=1,\dots, r\\ \g\in \S}} 
 ( \b_{j,\g}\cup \b_{j,\g}^a\cup \b_{j,\g}^b\cup   D_{\g})=\bigcup_{j=1,\dots,r}\left(\bigcup_{\g\in\S_j}  ( \b_{j,\g}\cup \b_{j,\g}^a\cup \b_{j,\g}^b\cup   D_{\g}) \right)=\bigcup_{j=1,\dots,r}\mathcal B_j,
 $$
 where we have denoted by $\mathcal B_j$ the set into brackets.  
 
 We claim that,  after step (II), the set $\mathcal B$ is contractible.
  Indeed, for each $j$, the set $ \cup_{\g\in\S_j\cap \S^2}\b_{j,\g}$ is contractible, by construction. Actually, it is homeomorphic to a tree. Adding to this set the curves $\b_{j,\g}^a,\b_{j,\g}^b$  for $\g\in\S_j\cap \S^1$ produce loops, but these loops are cut into arcs in step (I).  Thus, each set $\mathcal B_j$ is contractible. Finally, after step (II), the  sets $\mathcal B_j$  are joined among them following  the rules given by the tree $\mathcal T$, i.e., if $\O_j, \O_{j'}$ are connected in the tree, then $\mathcal B_j$ and $\mathcal B_{j'}$ share a single point. It follows that $\mathcal B$ is contractible.

Notice also that  $Q_2\menos \mathcal{B}$ is connected, because $\mathcal{B}$ intersects each boundary component of $Q_2$ in a connected set.

(III) This step is technical: we   cut along some curves, but these curves will not be in the cutting locus $W$. Cut    $Q_2$ along $\mathcal B$ to   obtain $(Q_3,b_3)$, where  $Q_3$   is   homeomorphic to a connected  orbifold with non-empty boundary, and    $b_3\colon Q_3\to Q_2$. 
Notice that:   
\begin{itemize}
\item if $Q_2$ has no boundary, then 
 $b_3^{-1}(\mathcal B)$ is the only boundary component of $Q_3$;
 \item on the other hand, if $Q_2$ has non-empty boundary, then  $b_3^{-1}(\mathcal B)$ is strictly contained in a boundary component of $Q_3$. This is because, in this case,  $\mathcal B$  has at least one point in $\partial Q_2$, since  $\partial Q_2$ corresponds to the cuttings done in steps (I) and (II), and these cuttings are along curves $\g\in \S$, which always contain the points $*_{\g}$.
 \end{itemize}
 In both cases let $Y_1$ the boundary component of $Q_3$ containing $b_3^{-1}(\mathcal B)$.
 
 (IV)
Now we want to cut $Q_3$ into a polygon, with  ${\rm Sing}\,Q_3$ in its boundary,  and avoiding $b_3^{-1}(\mathcal B)$. This can be easily done: first cut along non-separating curves $z_i$ to obtain an orbifold  $Q_4$ of genus 0. Let $Y_1,Y_2,\dots, Y_s$ be the set of boundary components and cone points   of $Q_4$. 
If $Y_1$ was  the only boundary component of $Q_3$, take simple  and disjoint (except at their endpoints) arcs $w_i, i=2,\dots, s$   joining the cone points  $Y_i, Y_{i+1}$. 
In the other case,  take $w_i, i=1,\dots, s$ as before, but only taking care that $w_1$ does not intersect  $b_3^{-1}(\mathcal B)$.

 (V) Let $W\subset P$ be the union of (the images under the $b_i$ of) the curves we have cut along in steps (I), (II) and (IV). Then cutting $\O$ along $W$ we obtain 
 $(\bar P, \bar b)$, where  $\bar P$ is a closed polygon  satisfying  (a). 
 Notice that $\bar b^{-1}$ is well defined and continuous when restricted to $\O\menos W$.

Finally, we will remove some parts of the boundary of $\bar P$ to obtain $(P,b)$ so that $b$ be   a bijection and so that all the conditions (a)-(d) be satisfied. Notice that (a), (c) and (d) will be automatically satisfied, since the arcs $\b_{j,\g},\b_{j,\g}^a, \b_{j,\g}^b$ only intersect $W$ in their final endpoint. Thus, we only need to care about (b). 

Notice that $W\cap \b_j=\emptyset$ for any $j=1,\dots, r$, because $\b_j$ only intersects curves $\g$ in the  
tree $\mathcal T$, and these curves are not intersected by $W$ (as seen in  (III)).  Thus, $b^{-1}\b_j$ will be  continuous, independently of which subset of $\partial \bar P$ we remove.    
On the other hand, take $\g\in \S^1$, and  notice that $W\cap \b_{\g}=\{*_{\g}\}$. Then $\bar b^{-1}(\b_{\g})$ contains two preimages of $*_{\g}$, one of them isolated. Then, we remove from $\partial \bar P$ the preimage of $\g$ containing the isolated preimage of $*_{\g}$.     We proceed in similar way with the  $\g\in \S^2$ such that $\g\subset W$ (if $\g\not\subset W$, then we need do nothing). Finally, we remove parts of the boundary of $\bar P$ in any way, in order to obtain a $(P,b)$ where now $b$ is a bijection. 
\end{proof}

\subsection{Notations concerning the covering and main result.} 
It will be key for Theorem \ref{thm:Combinatotial} a labeling for the points in the covering space $S$.
  \begin{itemize}
  \item[(1)] We fix a point $\tilde *$ in $p^{-1}(*)$. 
  
  \item[(2)] Let $\tilde b\colon P\to S$ be the unique lift of $b$ such that $\tilde b (b^{-1}(*))=\tilde *$, and let $\widetilde P=\tilde b(P)$. If $y\in\O$ we call $\tilde{y}$ the only point in $p^{-1}(y)\cap \widetilde{P}$.
  \item[(3)] If $\a$ is a path in $\O$ with origin in $*$, we denote by $\tilde{\a}$ its unique lift starting at $\tilde{*}$.
  \item[(4)] We have that $S=\cup_{g\in G} g\cdot \widetilde{P}$, and the union is disjoint except for the preimages of the cone points of $\O$ (notice that $e\cdot \widetilde{P}=\widetilde{P}$, where $e$ is the identity element of $G$). In this way, if $y\in \O$ is not a cone point, then $p^{-1}(y)=\{g\cdot \tilde{y}\,\colon\, g\in G\}$.
  
  \end{itemize}

 \begin{teo}\label{thm:Combinatotial}
Let $(S,G,\iota)$ be an action of the group $G$ on $S$, let $p\colon S\to \O$ be the associated branched covering and let  $\Phi\colon \pi_1(\O,*)\to G$ be the associated epimorphism.
\\
Let  $\Sigma=\{\g_1,\dots, \g_k\}$ be a multicurve   in $\O$ which decomposes $\O$ in the suborbifolds $\,\O_1,\dots, \O_r$,
and let  $\mathcal{G}_{\Sigma}$ be the stable graph   determined by $\S$ under $p$.
  We consider the homomorphisms $\Phi_j$ and $\Phi_{\g}$ and follow the notations as   explained above. 
  
   Then: 
 \begin{itemize}

 \item[(i)] (Number of vertices  of $\mathcal{G}_{\S}$.) 
 The number of vertices of  $\mathcal{G}_{\S}$ is equal to
 $$
 V =  
  \frac{|G|}{|{\rm Im}\, \Phi_1|} + \dots, +  \frac{|G|}{|{\rm Im}\, \Phi_r|}
 $$
The vertex of $\G_{\S}$  denoted  by $V_{j,g\Im \Phi_j}$ or by $V_{j,g}$ (for $j=1,\dots, r$, $g\in G$)  is   the component of $p^{-1}(\O_{j })$ which contains the points $g {\rm Im}\Phi_{j }\cdot \widetilde *_{j }$.
The map $V_{j,g}\mapsto (\O_j, g{\rm Im}\Phi_j  )$ gives a bijection between the set of vertices 
   of $\G_{\S}$ and  the set
 $
\{ 
  (\O_j, g{\rm Im}\Phi_j  )
 \; : \; j=1,\dots, r,  \quad   g\in G
 \}.
 $  
  \item[(ii)]  (Number of edges of $\mathcal{G}_{\S}$.)
The number of edges  of  $\mathcal{G}_{\S}$ is equal to
 $$
E =  
  \frac{|G|}{|{\rm Im}\, \Phi_{\g_1}|} + \dots, +  \frac{|G|}{|{\rm Im}\, \Phi_{\g_k}|}.
 $$  
The edge of $\G_{\S}$  denoted by $E_{\g,g\Im\Phi_{\g}}$ or by $E_{\g,g}$ 
  (for $\g\in\S$, $g\in G$)  is   the component of $p^{-1}(\g )$ which contains the points $g {\rm Im}\Phi_{\g}\cdot \widetilde *_{\g }$.
The map $E_{\g,g}\mapsto (\g, g{\rm Im}\Phi_{\g}  )$ gives a bijection between the set of edges  
and the set
 $
\{ 
 \left(\g_i, g{\rm Im} \Phi_{\g_i}  \right)
 \; : \; i=1,\dots, k,  \quad   g\in G
 \}.
 $  

 \item[(iii)] (Degrees of    vertices.) The degree of the vertex $V_{j,g}$ 
is equal to 
 $$
 D_{j}= |\Im \Phi_j| \left( \sum_{\g\in\S_j\cap\S^2}\frac{1}{|\Im\Phi_{\g}|} +2\sum_{\g\in\S_j\cap\S^1}\frac{1}{|\Im\Phi_{\g}|} \right).
  $$
 \item[(iv)]  (Weights of  vertices.) The weight $w^j$
 of the vertex $V_{j,g }$ is equal to   
$$\quad \quad
w^j
=1 
 -\frac{1 }{2}\left(|\Im \Phi_j|\, \chi(\O_j)+ D_j\right).  
$$ 
 
\item[(v)]  (Edges connecting vertices.) Let $E_{\g,g}$ be an edge, $\g\in \S$. 
\begin{itemize}
\item[Case 1.] If $\g\in\S_{j}\cap \S_{j'}$, for $j\not=j'$,  then $E_{\g,g}$ joins $V_{j_1, g\Phi(\b_{\g}\b_{j_1\g}^{-1}\b_{j_1}^{-1})}$ and  $V_{j_1, g\Phi(\b_{\g}\b_{j_2\g}^{-1}\b_{j_2}^{-1})}$ (notice that one of the paths  $\b_{\g}\b_{j_i\g}^{-1}\b_{j_i}^{-1}$, $i=1,2$ is trivial).
\item[Case 2.]  If $\g\in \S^1\cap \S_j$, then 
$E_{\g,g}$ joins  $V_{j, g }$ and  $V_{j, g\Phi(\b_{\g} (\b_{j,\g}^b)^{-1}\b_{j}^{-1})}$
( see Figure \ref{fig:CaminosBeta}).
  
\end{itemize}

 \end{itemize}
 \end{teo}
 
 






\begin{rem}
  Parts (i) to (iv) of the previous theorem depend only on the number of elements of the subgroups $\Im \Phi_j$ and $\Im \Phi_{\g}$ for all $j=1,\dots, r$ and $\g\in \S$. 
\end{rem}

\begin{proof} 
 (i) The vertices of $\G_{\S}$ are the connected components of $S\menos p^{-1}(\S)$. Each of these connected components is a covering of  one of the suborbifolds $\O_j$, $j=1,\dots, r$. Thus, we can also see the vertices of $\G_{\S}$ as the connected components of $p^{-1}(\O_j)$ for all $j=1,\dots, r$. Then, the result follows from
  Lemma \ref{lem:stabilizers}.

 \medskip 
 (ii) The number $E$ of edges of $\G_{\S}$ is the number of connected components of $p^{-1}(\S)$, which, in turn, is
  $E=\S_{i=1}^{k}n_{\g_i}$, where $n_{\g_i}$   
 is the number of connected components of   $p^{-1}(\g_i)$ (since the $\g_i$ are disjoint, their preimages are also disjoint). As in part (i), the result then follows from Lemma \ref{lem:stabilizers}.

  \medskip
  
  (iii) Let $V_{j,g}$ be a vertex of $\G_{\S}$.
  For each $\g\in\S$ that is contained in $\bar \O_j$, we consider a small closed tubular neighborhood $U_{\g}$ not intersecting any other cone point or  $\g_i$. Thus, if $\g$ is a closed curve,   $U_{\g}$ is homeomorphic to an annulus, while if $\g$ is an arc,  $U_{\g}$ is homeomorphic to a disc. 
  
Let $\g\in \S_j$, i.e., $\g\subset \bar\O_j$. We distinguish three cases. 
\begin{itemize}
\item[1.]   $\g\in  \S^2$, i.e., $\g$ is  closed and contained in  another $\bar \O_{j'}$, for $j'\not=j$.  In this case, only one boundary component of  $U_{\g}$ is contained in $\O_j$. We call this component $\g^1$.
\item[2.]   $\g\in \S^1$ and is a closed curve.   In this case, the two boundary components $\g^1,\g^2$ of $U_{\g}$ are contained in $\O_j$
\item[3.] $\g$ is an arc (thus, $\g\in \S^1$).   In this case, $U_{\g}$ has just one boundary component $\g^1$, which is contained in $\O_j$
\end{itemize}

 
  In this situation, the degree of $V_{j,g}$ is the number of connected components of 
  $$ \bigcup_{\g\in \S_j\cap \S^2 }(p|_{V_{j,g }})^{-1}(\g ^1)\cup 
  \bigcup_{ \substack{\g\in \S_j\cap \S^1\\ \g \hbox{\tiny{ closed curve}} }}\Big((p|_{V_{j,g }})^{-1}(\g ^1) \cup  (p|_{V_{j,g }})^{-1}(\g ^2)\Big) \cup
   \bigcup_{ \substack{\g\in \S_j\cap \S^1\\ \g \hbox{\tiny{ arc}} }}(p|_{V_{j,g }})^{-1}(\g ^1) 
 $$

  The number of components of $(p|_{V_{j,g}})^{-1}(\g^i)$ is computed using Lemma \ref{lem:stabilizers}(c) applied to the covering map $p|_{V_{j,g}}\colon V_{j,g}\to \O_j$ and taking $X=\g^i$. The automorphism group of this covering is the stabilizer of $V_{j,g}$, which, by Lemma \ref{lem:stabilizers}(a), is equal a conjugate of  $\Im \Phi_j$. 
 On the other hand, when $\g$ is a closed curve, then  $\g^i$ is homotopic to $\g$ and therefore $\Im\Phi_X=\Im\Phi_{\g}$. While, if $\g$ is an arc, then $\g^1$ is homotopic to $\g^a\g^b$ (recall the notation from Section \ref{sec:Orbifolds}), and therefore $|\Im\Phi_X|=\frac{1}{2}|\Im\Phi_{\g}|$. Putting all together, we obtain the desired result.

  \medskip

 (iv) The weight $w^j$
 of the vertex $V_{j,g }$ is determined by the regular branched covering 
  $p|_{V_{j,g}}\colon V_{j,g}\to \O_j$. Explicitly, if  
  the orbifold $\O_j$ has  signature $(\tau,c;m_1,\dots, m_{n_j})$,  then  
 $p|_{V_{j,g}} $  has $ |\Im \Phi_j|$ sheets and $n_j$ branched points  or orders $m_1,\dots, m_{n_j}$. By the Riemann-Hurwitz formula we have that
 $
 \chi(V_{j,g}) =|\Im \Phi_j|\,\chi(\O_j)
 .$
 Since by part (iii) we know the number of boundary components of $V_{j,g}$, the relation $\chi(V_{j,g})=2-2w^j-c$ between the Euler characteristic and the genus of a surface gives the desired result.

 \medskip
 
(v) The proof of this part is based in the following two observations. 
\begin{itemize}
\item[(a)] Let $\tilde \a$ be a path  in $S$ with $p\tilde \a(0)=x, p\tilde \a(1)=y  $. If the map  $b^{-1}  p  \tilde\a $ is continuous, and $\tilde \a(0)=g\cdot \tilde x$, then $\tilde\a(1)=g\cdot \tilde y$. 
\item[(b)] Let $\g\in\S$, $\g\subset\bar\O_j$.   Consider   $E_{\g,g}$,  which contains the point $g\cdot \tilde*_{\g}$. Suppose there is  a path $\tilde \a$ from $g'\cdot \tilde *_j$ to $g\cdot \tilde*_{\g}$ such that the restriction of the map $b^{-1}p\tilde \a  $   to the interval $[0,1)$ is continuous. Then $E_{\g,g}$
is in the boundary of $V_{j,g'} $.
\end{itemize}
If $b^{-1}  p  \tilde\a $ is continuous, then it is a path in $P$ whose lift to $S$ is $\tilde \a$. 
Then,  by the notations before the theorem, we have (a). For (b), 
since $b^{-1}p\tilde \a([0,1)) $ is connected and $\tilde \a(0)=g'\cdot \tilde *_j$, by (a) we have that   all the points $\tilde \a(t), t\in[0,1)$ are also labeled as $g'\cdot \tilde y$, where  for $y\in\{ p\tilde\a(t) \, :\, 0\leq t<1\}$.   Thus, all the points $\tilde \a(t), 0\leq t<1$ are in 
$V_{j,g'}$ and, since these points converge to $g\cdot \tilde *_{\g}$, the result follows.

 \medskip
 
Case 1:  $\g\in \S^2$, i.e.,  $\g\subset \bar \O_{j_1}\cap\bar \O_{j_2}$. Then, the edge $E_{\g,g}$ is in the boundary of a component  which projects onto $\O_{j_1}$ and another which projects onto $\O_{j_2}$. We need to know the precise labels of these components, and this can be done using the previous observations.
Indeed,
consider in $\O$ the loops  $\Delta_i=\b_{\g}\b_{j_i,\g}^{-1}\b_{j_i}^{-1}$ $i=1,2$ (notice that one of them is trivial, by the definition of $\b_{\g}$). Let $g_i=\Phi(\Delta_i)$, so that $\tilde \Delta_i(1) =g_i\cdot \tilde *$. Since $b^{-1}\b_{\g} $ is continuous (by Lemma \ref{lem:PolygonP}(b)) and $\tilde \Delta_i(0)=\tilde{*}$, by  (a) above we have that 
  $\tilde{\Delta_i}$ contains the point $\tilde*_{\g}$.
 Also, since the   $b^{-1}\b_{j_i}$ is continuous (by Lemma \ref{lem:PolygonP}(b)) and $\b_{j_i}$ joins $*$ and $*_{j_i}$,  then the lift of $\b_{j_i}$ contained in $\tilde \Delta_i$ joins the points $g_i\cdot \tilde *$ and $g_i\cdot \tilde *_{j_i}$. Thus, the  lift of $\b_{j_i,\g}$ contained in $\tilde \Delta_i$ goes from $g_i\cdot \tilde *_{j_i}$ to  $\tilde *_{\g}$. Since $(b^{-1}\b_{j_i,\g})|_{[0,1)} $ is continuous (by Lemma \ref{lem:PolygonP}(c)), then by (b) above, $E_{\g,e}$ is in the boundary of $V_{j_i,g_i}$. In terminology of the graph $\G_{\S}$, the edge $E_{\g,e}$ joins the vertices $V_{j_1,g_1}$,   $V_{j_2,g_2}$.

Now, for a general $g\in G$, $g(E_{\g,e})$ joins $g(V_{j_1,g_1})$ and  $g(V_{j_2,g_2})$, i.e., $ E_{\g,g}$ joins $V_{j_1,gg_1}$ and  $V_{j_2,gg_2}$, as stated.

 \medskip
 
Case 2:   $\g\in \S^1$, assume $\g\subset\bar\O_j$. 
 We proceed in a similar way as before.
 Consider the loops   $\Delta_1=\b_{\g}(\b_{j,\g}^a)^{-1}\b_{j}^{-1}$ (which is trivial, by the definition of $\b_{\g}$) and $\Delta_2=\b_{\g}(\b_{j,\g}^b)^{-1}\b_{j}^{-1}$. Let $g_i=\Phi(\Delta_i), i=1,2$ ($g_1$ is the identity). Then $\tilde \Delta_i$ starts at $\tilde *$ and ends at $g_i\cdot\tilde{*}$ and contains one preimage of $*_{\g}$ and one preimage of $*_j$. Since $b^{-1}\b_{\g} $ is continuous, the first preimage is $\tilde *_{\g}$. And since $b^{-1}\b_j$ is continuous, then the second preimage is $g_i\cdot \tilde *_j$. Then, $g_i\cdot \tilde *_j$  is joined to $\tilde *_{\g}$    through a lift of the path   $\b_{j,\g}^a $, for $i=1$, or the path $\b_{j,\g}^b $, for $i=2$. Since $(b^{-1}\b_{j,\g}^a)|_{[0,1)} $ and  $(b^{-1}\b_{j,\g}^b )|_{[0,1)}$ are continuous 
  (by Lemma \ref{lem:PolygonP}(c)(d)),   by (b) above we have  that   
  $E_{\g,e}$ bounds the components $V_{j,g_1}$ and $V_{j,g_2}$. For general $g\in G$,  $E_{\g,g}$ bounds the components $V_{j,gg_1}$ and $V_{j,gg_2}$, as is claimed in the statement.
  \end{proof}

 Detailed examples of application of this theorem can be seen in Section \ref{sec:Pyramids}.

\section{Examples: Pyramid families}\label{sec:Pyramids}

\subsection{Pyramidal action.}
In this section   we consider the 2-dimensional equisymmetric stratum of   pyramidal hyperbolic surfaces $\mathcal P_n\subset \mathcal B_n$, which we define below. This stratum was also studied in \cite{victor} in order to  to determine the Riemann period matrices of the  jacobians of its elements. Here, we work out the characterization of 
the   strata of $\widehat{\M_n}$   intersected by $\widehat{\mathcal P_n}$. This is done  by mainly using   Theorem 5.1.

Let $G=D_n$ be the dihedral group of order $2n$
with presentation 
$$
D_n=\langle \rho, \sigma \;: \; \rho^n=\sigma^2=(\sigma\rho)^2=1 \rangle.
$$ 
 We will call {\it rotations} of the dihedral group to the elements of the subgroup  $\langle \rho\rangle$  and {\it symmetries} to the remaining elements. We recall that a  rotation and a  symmetry are never conjugate in $D_n$. Also recall that if $n$ is odd, all the symmetries are conjugate, while if $n$ is even there are two different conjugate  classes of symmetries.
 
 Let $S$ be a surface of genus $n\geq 3$. We consider the action of $D_n$ on $S$ given  in the Figure \ref{fig:Piramides}. That is, we imagine $S$ as the surface surrounding the edges of the projection of a regular $n$-pyramid onto its base. Then the action of $\rho$ is given by the restriction of the Euclidean rotation of order $n$  with axis the axis of the pyramid. And the action of $\sigma$ is given by a rotation of order 2 whose axis is a symmetry axis of the $n$-agonal base of the pyramid. 
 
 The quotient orbifold $\O=S/D_n$ has signature $(0;2,2,2,2,n)$. Let $P_1, \dots, P_4$ be the  cone points of order 2 and $P_5$ the cone point of order $n$, and let $x_i$ be a loop surrounding $P_i$ (see Figure \ref{fig:Piramides}). Then the fundamental group of this orbifold has presentation 
 $$
 \pi_1(\O)=\langle x_1,x_2,x_3,x_4,x_5 \;: \;  x_1^2=x_2^2=x_3^2=x_4^2=x_5^n=x_1x_2x_3x_4x_5=1 \rangle.
 $$

  An epimorphism which determines this action is $\Phi\colon \pi_1(\O)\to D_n$  defined as 
  $$\Phi(x_1)=\sigma, \quad \Phi(x_2)=\Phi(x_3)=\Phi(x_4)=\rho\sigma, \quad \Phi(x_5)=\rho. $$

\psfrag{P1}{$P_1$}
\psfrag{P2}{$P_2$}
\psfrag{P3}{$P_3$}
\psfrag{P4}{$P_4$}
\psfrag{P5}{$P_5$}

\psfrag{e}{$e$}
\psfrag{s}{$\sigma$}
\psfrag{r}{$\rho$}
\psfrag{r2}{$\rho^2$}
\psfrag{r3}{$\rho^3$}
\psfrag{rs}{$\rho\sigma$}
\psfrag{r2s}{$\rho^2\sigma$}
\psfrag{r3s}{$\rho^3\sigma$}

\psfrag{2pn}{$2\pi n$}
\psfrag{*}{$*$}
\psfrag{*g}{$*_{\gamma}$}
\psfrag{g}{$\gamma$}
\psfrag{*}{$*$}
 \psfrag{*g}{$*_{\gamma}$}
  \psfrag{bg}{$\beta_{\gamma}$}
  
\psfrag{O1}{$\mathcal{O}$}
 
 \psfrag{x1}{$x_1$}
  \psfrag{x2}{$x_2$}
   \psfrag{x3}{$x_3$}
    \psfrag{x4}{$x_4$}
     \psfrag{x5}{$x_5$}
 
\begin{figure} 
\center
\includegraphics[height=8cm,width=15cm]{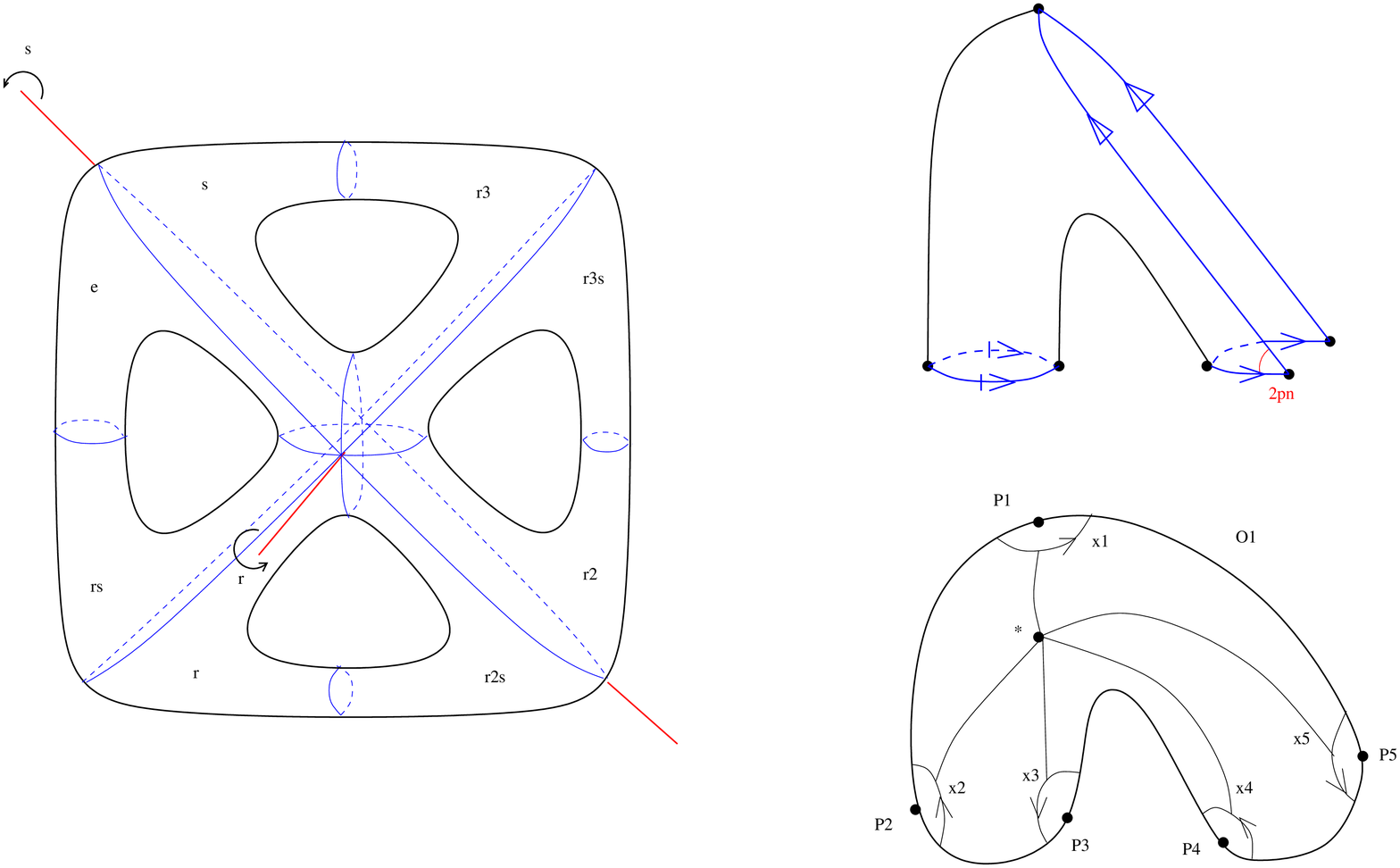}
\caption{Pyramids}
\label{fig:Piramides}
\end{figure}

This action will called the {\it pyramidal} topological action. 
In our notation, the  equisymmetric locus determined by this   action is
$\mathcal  M_n(D_n,\O,\Phi)$. In \cite{victor}, for any $n \geq  3$, a Fuchsian uniformization for the family $\mathcal P_n$ is given with another epimorphisms $\theta$. Nevertheless
a calculation proves that $\theta$ is in the same topological equivalence class of $\Phi$. Therefore the
2-dimensional family $\mathcal P_n$ coincides with our stratum $\mathcal  M_n(D_n,\O,\Phi)$. For brevity, in what follows, we will use the notation  $\mathcal P_n$.

\noindent{\bf Notation.} Extending the notation for the dihedral groups, we will use $D_1$ to denote the group generated by one reflection, which is    isomorphic to $\Z_2$, and $D_2$ for the group generated by reflections in two  orthogonal lines, which is isomorphic to $\Z_2\times \Z_2$.

\begin{rem}
 There is only one topological equivalence class of actions of $D_3$ on a hyperbolic
surface $S$ of genus $n= 3$ (see \cite{Broughton-Classifying}), and this action is topologically equivalent to the pyramidal action.
There are two different topological equivalence classes of actions of $D_4$ on a Riemann surface of genus $n= 4$ (see \cite{kimura}). It can be verified by a calculation that one of them is topologically
equivalent to the pyramidal action. 
\end{rem}

\subsection{Multicurves in $\O$ and first examples.}
We will consider any (admissible) multicurve $\S$ in $\O$ and, using Theorem \ref{thm:Combinatotial}, we will find the stable graph $\G_{\S}$ determined by $\S$ under $p$. 
In the orbifold $\O$, with signature $(0;2,2,2,2,n)$, it is easy to see that  there are only four types of admissible multicurves, up  to homeomorphism of $\O$ taking one to the other: (1) $\S=\{\g\}$ with $\g$ an arc; (2) $\S=\{\g_1,\g_2\}$ with both $\g_1,\g_2$  arcs; (3) $\S=\{\g\}$ with $\g$ a closed curve which separates three cone points of order 2 from the other cone points; and (4) $\S=\{\g_1,\g_2\}$ with  $\g_1$ an arc and $\g_2$ a closed curve as in (3). 

We start by illustrating in detail how Theorem \ref{thm:Combinatotial} is used in two concrete examples of multicurves of $\O$. 
  
  \subsection*{Example 1.}\label{sec:example1}
Consider the multicurve $\Sigma=\{\g\}$, where $\g$ is the arc joining the cone points $P_3,P_4$ shown  in the top-left figure of  Figure \ref{fig:Ejemplos1-2-3}. The complement $\O\menos \S$ consists  only of one suborbifold $\O_1$.  
We choose the basepoints $*_1=*$ and $*_{\g}\in \g$ and the path $\b_{1,\g}$  as  shown that  figure. 
  
 The suborbifold $\O_1$ is an (open) disc with three cone points of orders $2,2,n$.  Its fundamental group   is the subgroup of $\pi_1(\O,*)$ generated by $x_1,x_2,x_5$. Then, $\Im \Phi_1=\langle \sigma,\rho\rangle=G$,  and by Theorem \ref{thm:Combinatotial}(i), the number of vertices of $\G_{\S}$ is 1.

  On the other hand, recalling the notations of Section \ref{sec:Orbifolds} and of Section \ref{sec:InducedHomo}, we have that 
  the fundamental group of $\g$ (as a 1-dimensional orbifold) is generated by the loops $\g^a,\g^b$  and $i_*(\g^a),i_*(\g^b)$ are homotopic to $x_3,x_4$.  Thus, $ \Im\Phi_{\g}=\langle \rho\sigma\rangle $, and $|\Im\Phi_{\g}|=2$. Therefore, by 
   Theorem \ref{thm:Combinatotial}(ii), the number of edges of $\G_{\S}$ is $ n$. 
   
   Theorem \ref{thm:Combinatotial}(iii) shows that the degree of the only vertex $V_{1,e}$ is equal to $2n(2\frac{1}{2})=2n$ (of course since there is only one vertex, its degree must be twice the number of edges).
   
  Finally, by part (iv) of that theorem, and taking into account that $\chi(\O_1)=-1+\frac{1}{n}$, the weight  of the vertex is 
  $$1-\frac{1}{2}\left(2n (-1+\frac{1}{n})+ 2n\right)=0.
  $$
Thus, $\G_{\S}$ is the stable graph consisting on one vertex with weight $0$ and $n$ loops (see Figure \ref{fig:graphs} (1)).

\psfrag{*}{$*$}

\psfrag{*j}{$*_{j}$}

\psfrag{P1}{$P_1$}
\psfrag{P2}{$P_2$}
\psfrag{P3}{$P_3$}
\psfrag{P4}{$P_4$}
\psfrag{P5}{$P_5$}

\psfrag{e}{$e$}
\psfrag{s}{$\sigma$}
\psfrag{r}{$\rho$}
\psfrag{r2}{$\rho^2$}
\psfrag{r3}{$\rho^3$}
\psfrag{rs}{$\rho\sigma$}
\psfrag{r2s}{$\rho^2\sigma$}
\psfrag{r3s}{$\rho^3\sigma$}

\psfrag{2pn}{$2\pi n$}
\psfrag{*}{$*$}
\psfrag{*g}{$*_{\gamma}$}
\psfrag{g}{$\gamma$}
 
\psfrag{O1}{$\mathcal{O}_1$}
 
 \psfrag{x1}{$x_1$}
  \psfrag{x2}{$x_2$}
   \psfrag{x3}{$x_3$}
    \psfrag{x4}{$x_4$}
     \psfrag{x5}{$x_5$}
 \psfrag{bg}{$\beta_{\gamma}$}

\begin{figure}
\center
\includegraphics[height=10cm,width=17cm]{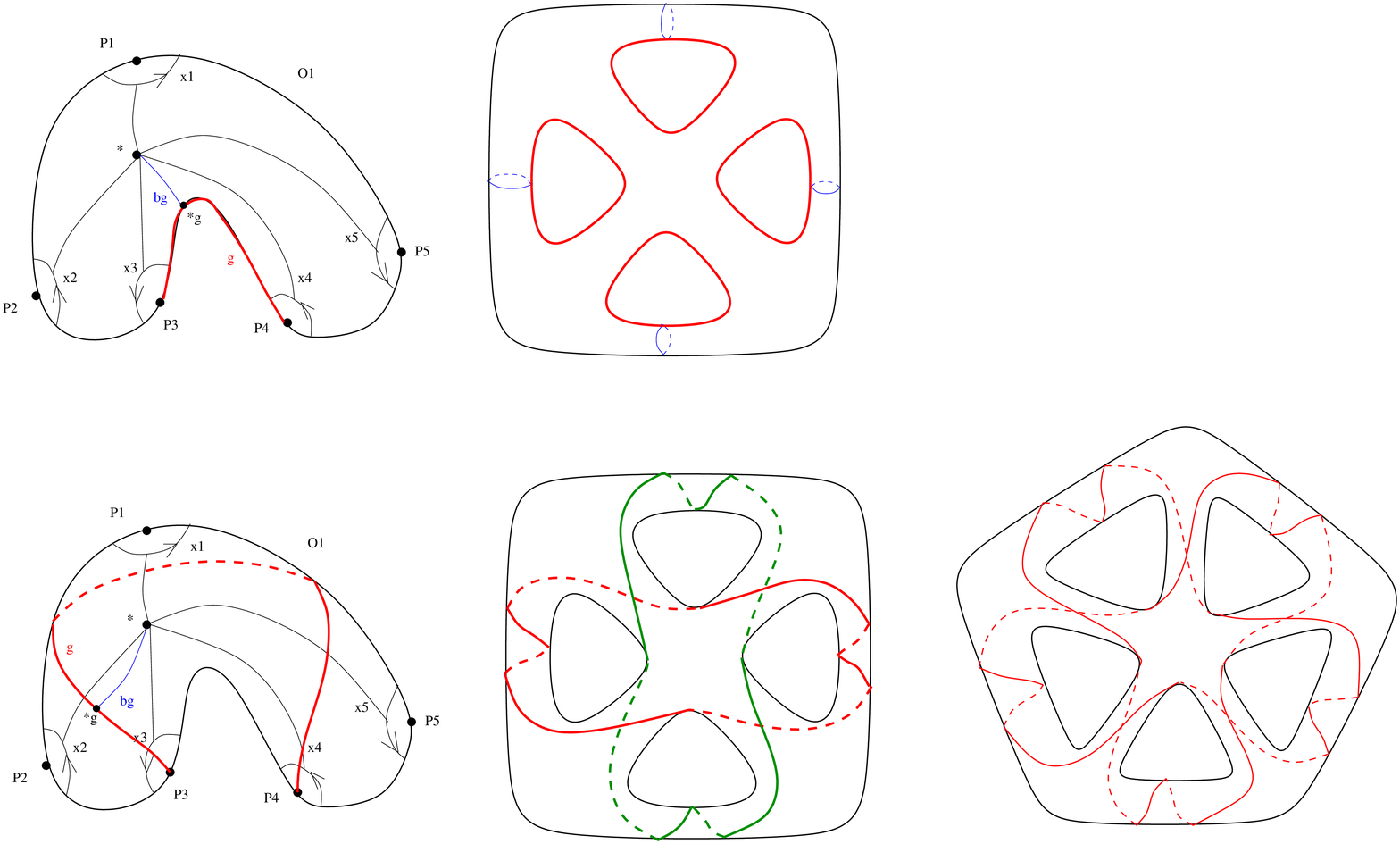}
 \caption{Examples 1 and 2}
 \label{fig:Ejemplos1-2-3}
 \end{figure}

  \subsection*{Example 2.}\label{sec:example2} 
  Consider the multicurve $\Sigma=\{\g\}$, where $\g$ is the arc joining the cone points $P_3,P_4$ as in the bottom-left figure  of  Figure \ref{fig:Ejemplos1-2-3}. As before, the complement $\O\menos \S$ consists  only of one suborbifold $\O_1$, which is a disc with three cone points of orders $2,2,n$. The 
  fundamental group of $\O_1$ is now  generated by the loops
 $ x_1, x_3^{-1}x_2x_3, x_4x_5x_4^{-1}  $. Then, $\Im\Phi_1$ is the subgroup of $G$ generated by $ \sigma, \rho\sigma, \rho\sigma\rho^2\sigma$, which is again the whole group $G$. 
 
 On the other hand, the generators $\g^a,\g^b$ of  $\pi_1(\g,*_{\g})$ are mapped by $i_*$ to  $x_3, x_1^{-1}x_4x_1$. Therefore 
 $$\Im\Phi_{\g}=\langle \rho\sigma, \sigma\rho\rangle =
 \langle \rho\sigma, \rho^2\rangle.
 $$
Notice that this group is again the whole group $G$ if $n$ is odd, but is a subgroup of index 2 of $G$ if $n$ is even. 

According to that, and applying Theorem \ref{thm:Combinatotial}, we have the following.
\begin{description}
\item[$n$ odd] the graph  $\G_{\S}$ has $\frac{2n}{2n} =1$ vertex, $\frac{2n}{2n} =1$ edge, and the vertex has degree $2n\frac{2}{2n}=2$ and weight $1-\frac{1}{2}\left(2n (-1+\frac{1}{n}) + 2\right) =n-1$.
\item[$n$ even]  the graph  $\G_{\S}$ has $\frac{2n}{2n} =1$ vertex, $\frac{2n}{n} =2$ edges, and the vertex has degree $2n\frac{2}{n}=4$ and weight $1-\frac{1}{2}\left(2n(-1+\frac{1}{n})+4\right) =n-2$.
\end{description}
In Figure \ref{fig:Ejemplos1-2-3}   we can see pictures of the preimage of $\S$ in example 1 and example 2 for $n=4$ and for $n=5$. By pinching these curves, we obtain  
 the stable surfaces corresponding to the stable graphs $\G_{\S}$.

\subsection{Multicurves $\S$ consisting in one arc.}
Theorem \ref{thm:Piramides1arco} below will show  all the stable graphs determined by $\S$ under $p$ when $\S$ is a multicurve consisting only of one arc $\g$.  
Examples 1 and 2 above are two cases of this theorem. Next lemma will be useful in that theorem. 

\begin{lem}\label{lem:ImPhi_arco}
Let   $\g$ be a simple arc in $\O$  joining two cone points of order 2
and let $\g^a,\g^b$ be the generators of $\pi_1(\g,*_{\g})$ (recall notation in Section \ref{sec:Orbifolds}). Then $\Phi_{\g}(\g^a\g^b)=\rho^k$ for some $k=1,\dots, n$ and $\Im\Phi_{\g}$ is a dihedral subgroup of $D_n$ isomorphic to $D_m$, with $m=\frac{n}{(n,k)}$, where $(n,k)$ denotes the greatest common divisor of $n,k$.
 
 Conversely, given any $m$ dividing $n$, there is an arc $\g$ so that $\Im \Phi_{\g}$ is isomorphic to $D_m$.
\end{lem}

\begin{proof} 
 Assuming $\g$ joins the cone points $P_{i_1},P_{i_2}$, $i_1,i_2=1,\dots, 4$,  then the loops $\b_{\g}\g^a\b_{\g}^{-1}$, $ \b_{\g}\g^b\b_{\g}^{-1}$ are conjugate to $x_{i_1}$ and $x_{i_2}$. Since $\Phi(x_{i_1}) , \Phi(x_{i_2})\in \{\sigma, \rho \sigma\}$,  we have that $\Phi_{\g}(\g^a)$, $ \Phi_{\g}(\g^b)$ are both conjugate to $\sigma$ or $\rho\sigma$; thus, they are symmetries of $D_n$ and their product is a rotation, say $\rho^k$, which has order $m=\frac{n}{(n,k)}$.   Hence, 
$$
\Im\Phi_{\g}=\langle\Phi_{\g}(\g^a), \Phi_{\g}(\g^b)\rangle =\langle \Phi_{\g}(\g^a), \Phi_{\g}(\g^a\g^b ) \rangle \cong D_m. 
$$ 

For the converse, we consider the following 
  two families of examples of $\S=\{\g\}$, where  $\g$ is shown on the bottom of  Figure \ref{fig:Piramides1arco}. On the top of that figure, we show another view of the orbifold $\O_1$. 
\begin{description}
\item[Family  on the bottom left] The arc $\g$ starts at $P_3$,   wraps $k$ times  around $P_1,P_4$ and ends at $P_4$, with $k\geq 0$. The fundamental group of the orbifold $\g$ is generated by the two loops $\g^a,\g^b$. Then
\begin{eqnarray*}
\Phi_{\g}(\g^a)&=&\Phi(\b_{\g}\g^a\b_{\g}^{-1})= \Phi(x_4)
=\rho\sigma
\\
\Phi_{\g}(\g^b)&=&\Phi(\b_{\g}\g^b\b_{\g}^{-1})=
\Phi\left(( x_1 x_4)^k x_1 x_3 x_1^{-1} ( x_1 x_4 )^{-k}\right)
\\ &=&
(\sigma\rho\sigma)^k\sigma \rho\sigma\sigma(\sigma\rho\sigma)^{-k}=\sigma\rho^{k+1}\sigma\rho^{-k}\sigma = \sigma\rho^{2k+1}.
\end{eqnarray*}

\psfrag{*}{$*$}
 \psfrag{*g}{$*_{\gamma}$}
  \psfrag{bg}{$\beta_{\gamma}$}
  
\psfrag{O1}{$\mathcal{O}$}

 \begin{figure}
\center 
\includegraphics[height=8cm,width=10cm]{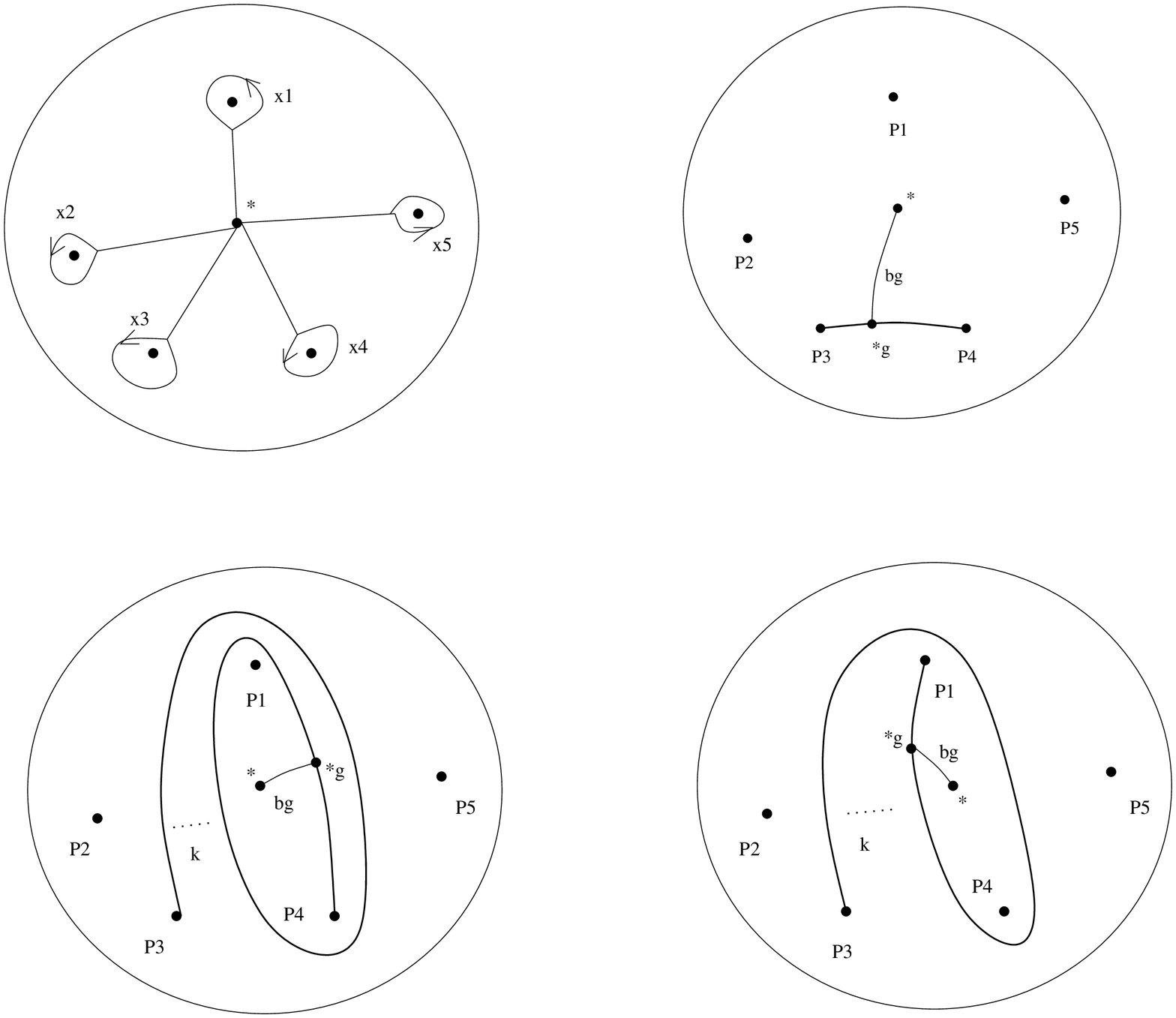}
\caption{}
\label{fig:Piramides1arco}
 \end{figure}

In this way, we have
 $$
 \Im\Phi_{\g} =\langle \rho\sigma, \sigma\rho^{2k+1}\rangle =\langle \rho\sigma, \rho^{2(k+1)}\rangle.
 $$

\item[Case in the top right] Notice that the previous family does not include the case on this figure, so we consider it apart. In this case, we have
\begin{eqnarray*}
\Phi_{\g}(\g^a)&=&\Phi(\b_{\g}\g^a\b_{\g}^{-1})= \Phi(x_3)=\rho\sigma
\\
\Phi_{\g}(\g^b) &=&\Phi(\b_{\g}\g^b\b_{\g}^{-1})=
\Phi(x_4^{-1}) = \rho\sigma,
\end{eqnarray*}
and thus 
 $$
 \Im\Phi_{\g} =\langle \rho\sigma, \rho\sigma \rangle =\{ \rho\sigma,1 \}\cong D_1  .
 $$
 
\item[Family on the bottom right] The arc $\g$ starts at $P_3$,   wraps $k\geq 0$ times  around $P_1,P_4$ and ends at $P_1$. As before, $\pi_1(\g,*_{\g})=\langle\g^a,\g^b\rangle$, and in this case we have 
 \begin{eqnarray*}
\Phi_{\g}(\g^a)&=&\Phi(\b_{\g}\g^a\b_{\g}^{-1})= \Phi(x_1)= \sigma
\\
\Phi_{\g}(\g^b)&=&\Phi(\b_{\g}\g^b\b_{\g}^{-1})=
\Phi\left((x_4 x_1  )^k x_3 (x_4 x_1  )^{-k}\right)
\\ &=&
(\rho\sigma\sigma)^k\rho\sigma  (\rho\sigma\sigma)^{-k}= \rho^{k }\rho\sigma \rho^{-k}  =\rho^{2k+1} \sigma .
\end{eqnarray*}
In this way, we have
 $$
 \Im\Phi_{\g} =\langle  \sigma, \rho^{2k+1}\sigma\rangle =\langle  \sigma, \rho^{2k+1}\rangle.
 $$
\end{description}

The three cases above show that for any $d=1,\dots, n$, there is $\g$ so that $\Im\Phi_{\g} $  is isomorphic to the dihedral group $D_{\frac{n}{(n,d)}}$. Now, given $m>0$ dividing $n$, just take $d=\frac{n}{m}$ and we obtain the result. 
\end{proof}

\begin{teo}\label{thm:Piramides1arco}
The multicurve $\S$ consists of exactly one arc if and only if  there is number   $m\geq 1$ dividing   $n$ such that the  stable graph $\G_{\S}$ has one vertex and $n/m$ edges (thus, the degree of the vertex is $2n/m$ and its weight is $n-n/m$, see Figure \ref{fig:graphs}(1)). 
\end{teo}

\psfrag{1}{$1$}
\psfrag{0}{$0$}
\psfrag{n/m}{$\frac{n}{m}$}
\psfrag{n-n/m}{$n-\frac{n}{m}$}
\psfrag{d}{$\vdots$}
\psfrag{w}{$w$}
\psfrag{(1)}{(1) $\S=$ one arc}
\psfrag{(2)}{(2) $\S=$ two arcs} 
\psfrag{(3)}{(3) $\S=$ one closed curve}
\psfrag{(4)}{(4) $\S=$ one arc, one closed curve}
 
 \begin{figure}
\center 
\includegraphics[height=9cm,width=14cm]{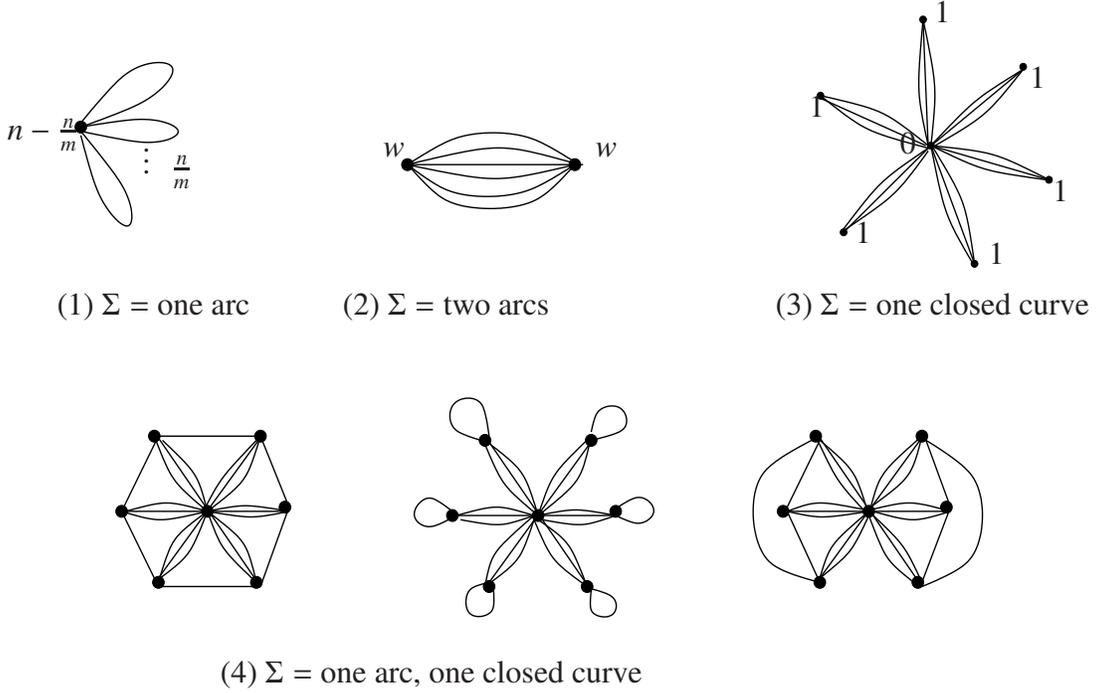}
\caption{Types of graphs appearing in Theorems \ref{thm:Piramides1arco}, \ref{thm:Piramides2arcos}, \ref{thm:Piramides1curva}, \ref{thm:Piramides1arc1curve}. No number assigned to a vertex means weight equal zero.}
\label{fig:graphs}
 \end{figure}

\begin{proof} Let $\S=\{\g\}$ where $\g$ is an arc. 
As in   examples 1, 2 above, the complement of $\S$ is just one orbifold $\O_1$, which is a disc with three cone points of orders 2, 2, and $n$. The fundamental group of $\O_1$ is generated by a a conjugate of $x_5$ and conjugates of $x_i,x_j$, where $x_i,x_j$ are the loops surrounding the two cone points of order 2 in $\O_1$. Thus, $\Im\Phi_1$ is generated by $y_1,y_2,y_3$, where $y_1$  a conjugate of $\Phi(x_5)$ and $y_2,y_3$ are  conjugates of $\Phi(x_i),\Phi(x_j)$. 
Now, $\Phi(x_5)=\rho$, which has order $n$, and thus $y_1$ has also  order $n$. On the other hand, $\Phi(x_i),\Phi(x_j)$ are equal to one of  $\sigma, \rho\sigma $, which are symmetries of the dihedral group $D_n$. Therefore, $y_2,y_3$ are also symmetries of $D_n$. In conclusion, $\Im\Phi_1$ is generated by an element of order $n$ and a symmetry, and thus, it is the whole group $D_n$. Applying Theorem \ref{thm:Combinatotial}(i), we have that $\G_{\S}$ has   one vertex.  

On the other hand, by Lemma \ref{lem:ImPhi_arco}, $\Im\Phi_{\g}$ is isomorphic to some $D_m$. Then, by Theorem \ref{thm:Combinatotial}(ii), we have that the number of edges is equal to $\frac{2n}{2m}$, which proves the desired result.

\medskip

Conversely, suppose  $\G_{\S}$ is a graph as stated. Since it has only one vertex, then $\O\menos\S$ consists just of one suborbifold. Then, it is either $\S=\{\g\}$ or $\S=\{\g_1,\g_2\}$.  Theorem \ref{thm:Piramides2arcos} below shows that the second possibility is impossible since it  produces a graph with  two vertices. Now, given $m>0$ dividing $n$, by Lemma \ref{lem:ImPhi_arco}, there is $\g$ so that $\Im\Phi_{\g}$ is isomorphic to $D_m$. 
Taking $\S=\{\g\}$, and using Theorem \ref{thm:Combinatotial},    the graph $\G_{\S}$ has $\frac{2n}{2d}$ edges, and the result follows. 
\end{proof}

\subsection{Multicurves $\S$ consisting in two arcs.}

 \begin{teo}\label{thm:Piramides2arcos}
(a) Suppose $\S=\{\g_1,\g_2\}$ where  $\g_1,\g_2$ are both arcs. Then $\G_{\S}$ is the stable graph consisting on two vertices of the same weight and $E$ edges joining one vertex  to the other, where $E=(n,k)+(n,k+1)$ for some $k=1,\dots, n$ (see Figure \ref{fig:graphs}(2)). 

(b) Moreover,  if $\G$ is a stable graph of genus $n$  with two vertices of equal weight and $(n,k)+(n,k+1)$ edges joining the two vertices (for some $k=1,\dots, n$), then there is a multicurve $\S=\{\g_1,\g_2\}$ where the $\g_i$ are arcs, so that $\G=\G_{\Sigma}$.
 \end{teo}
 \begin{proof}
(a) Assume $\g_1$ joins the cone points $P_{i_1},P_{i_2}$ and  $\g_2$ joins the cone points $P_{i_3},P_{i_4}$, where  $(i_1,i_2,i_3, i_4)$ is a permutation of $(1,2,3,4)$.   

The complement $\O\menos \S$ is an orbifold $\O_1$ which is an annulus with a cone point of order $n$. 
The fundamental group $\pi_1(\O_1,*)$ is generated by $z_1 x_5 z_1^{-1} $ and $ z_2 x_{i_1}z_2^{-1} z_3x_{i_2}z_3^{-1}$ for some $z_1,z_2,z_3\in\pi_1(\O,*)$. Then, $\Im\Phi_1$ is generated by a conjugate of $\Phi(x_5)$, which is an element of order $n$, and by  $\Phi( z_2 x_{i_1}z_2^{-1})\Phi( z_3x_{i_2}z_3^{-1} )$, the product of two symmetries, that is a rotation. Thus, $\Im\Phi_1=\langle\rho\rangle$,   $| \Im\Phi_1|=n$ and  the number of vertices of $\G_{\S}$ is $2n/n=2$. 

On the other hand, by Lemma \ref{lem:ImPhi_arco},  $\Im\Phi_{\g_1},\Im\Phi_{\g_2}$ are both dihedral subgroups of $D_n$, say 
$$
\Im\Phi_{\g_1}= \langle \Phi_{\g_1}(\g_1^a), \Phi_{\g_1}(\g_1^a\g_1^b)\rangle \cong D_{m_1},\quad 
\Im\Phi_{\g_2}= \langle \Phi_{\g_2}(\g_2^a), \Phi_{\g_2}(\g_2^a\g_2^b)\rangle \cong D_{m_2},
$$ 
for some $m_1,m_2$ dividing $n$.
Now, calling $\g'_i=\b_{\g_i}\g_i^a\g_i^b\b_{\g_i}^{-1}, i=1,2$, and   looking at the fundamental group  of $\O_1$, we have the equality   $\g_2'= \g_1'(z_1 x_5z_1^{-1})$,    up to the orientation of $\g_1', \g_2'$. Thus,
$$
\Phi(\g_2')= \Phi(\g_1')\Phi(z_1 x_5z_1^{-1})=\Phi(\g_1')\,\rho^{\pm1}
$$
where we have used that $\Phi(x_5)=\rho$ and any conjugate of $\rho$ is equal to $\rho^{\pm 1}$. Now, we know that   $\Phi(\g_1'), \Phi(\g_2')$ are rotations of $D_n$ (since they are product of two symmetries), say $\Phi(\g_1')=\rho^k$, for $k$ any integer modulo $n$. Thus, $\Phi(\g_2')=\rho^{k\pm 1}$. Since $\Im\Phi_{\g_i}\cong D_{m_i}$, we have that the order of $\g_i'$ is $m_i$, i.e., $\frac{n}{(n,k)}=m_1$ and  $\frac{n}{(n,k\pm 1)}=m_2$. 
Then, by Theorem \ref{thm:Combinatotial}(ii), we have that the number of edges of $\G_{\S}$ is equal to $(n,k)+(n,k\pm 1)$ and the result follows.

The two vertices have the same weight and the same degree because both vertices cover the same suborbifold $\O_1$.

It is left to show that each edge joins the two vertices, for which, we will use Theorem \ref{thm:Combinatotial}(v). By part (i) of that theorem and since $\Im\Phi_1=\langle\rho\rangle$,  the two vertices $V_{1,e},V_{1,\sigma}$ of $\G_{\S}$  are in correspondence with the cosets $\langle \rho\rangle, \sigma\langle\rho\rangle$. Let $E_{\g_i,g}$ be an  edge of $\G_{\S}$, which is in correspondence with the coset $g\Im\Phi_{\g_i}$. Since $\Phi_{\g_i}$ is a dihedral subgroup, it contains rotations and symmetries of $D_n$, and so does any coset $g\Im\Phi_{\g_i}$. So we can assume that $g$ is a rotation. Now, 
with the notations of Section \ref{sec:Multicurve}, we 
know that $\b_{\g_i}(\b_{1,\g_i}^b)^{-1}\b_1^{-1}$ is freely homotopic to either $\g_i^a$ or $\g_i^b$, and thus
  $\Phi(\b_{\g_i}(\b_{1,\g_i}^b)^{-1}\b_1^{-1})$ is a symmetry of $D_n$.
 Thus, by the theorem, the edge  $E_{\g_i,g}$ joins the vertices $V_{1,e}$ and $V_{1,\sigma}$

\psfrag{*g1}{$*_{\g_1}$}
\psfrag{*g2}{$*_{\g_2}$}
\psfrag{g1}{$\g_1$}
\psfrag{t}{$t$}
\psfrag{g2}{$\g_2$}
\psfrag{1=rho}{$\Phi(\g_1)=\rho$}
\psfrag{2=rho2}{$\Phi(\g_2)=\rho^2$} 
\psfrag{1=1}{$\Phi(\g_1)=1$}
\psfrag{2=rho}{$\Phi(\g_2)=\rho$}
\psfrag{t=0}{$t=0$}
\psfrag{1,rho}{$\Phi(\g_1)=1, \Phi(\g_2)=\rho  $}

 \begin{figure}
\center 
\includegraphics[height=8cm,width=14cm]{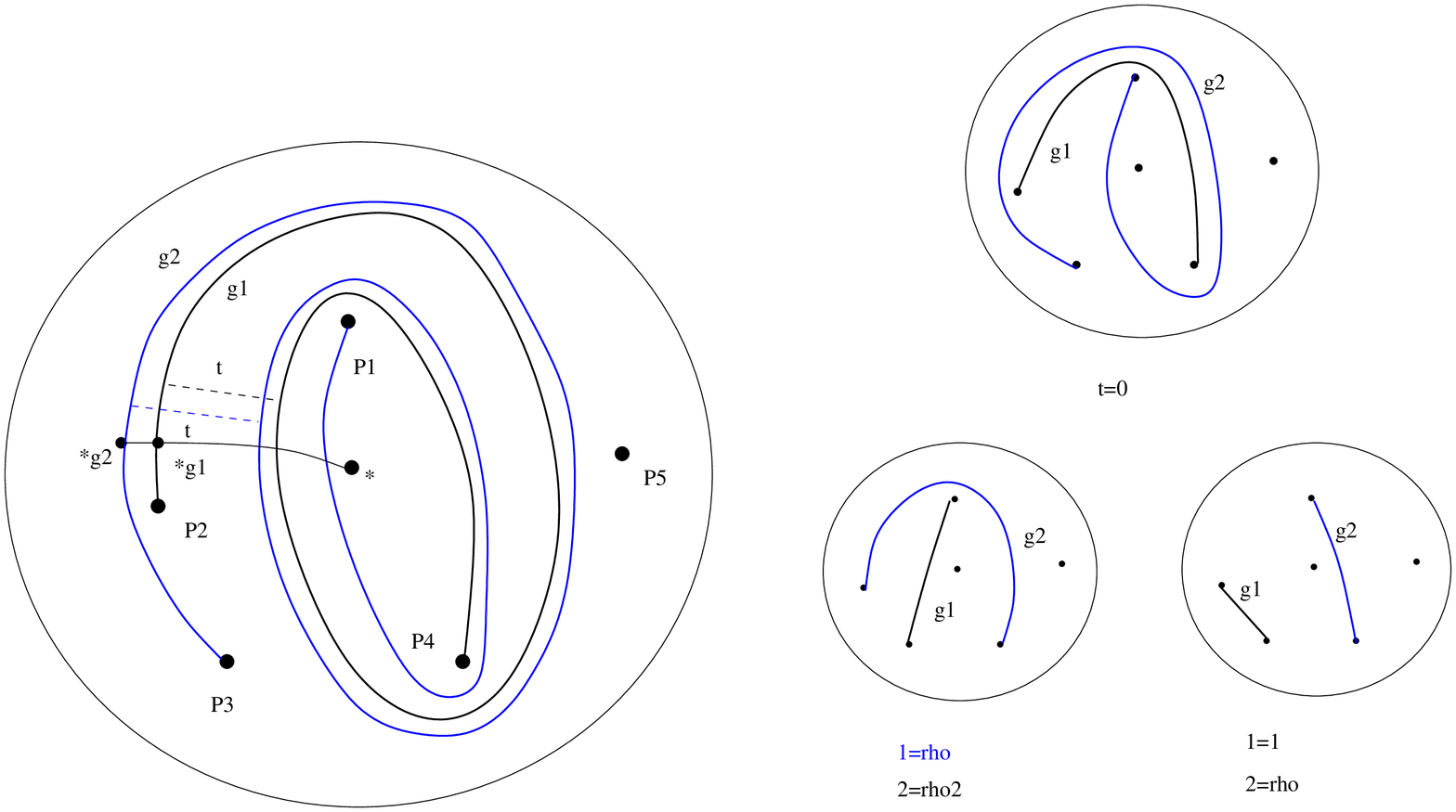}
\caption{}
\label{fig:Piramides2arcos}
 \end{figure}

\medskip
(b) We    take    $\g_1, \g_2$ as in  Figure \ref{fig:Piramides2arcos}, where both arcs wrap $t\geq 0$ times around the points $P_1,P_4$. Then we can compute that  $\Phi(\g_1')=\rho^{2t +2}$ and  $\Phi(\g_2')=\rho^{2t +3}$  so that $2t+2,2t+3= k,k+1$ for all $k\geq 2$. Since $(n,k)=(n,k+an), a>0$, then the above is enough to prove the remaining cases $k=1,2$. Nevertheless, Figure \ref{fig:Piramides2arcos} shows  also alternative arcs to obtain these cases. 
 \end{proof}
 
 \noindent{\bf Examples.}
 \begin{enumerate}
 \item For each $n$, there is always a multicurve $\S$ so that $\G_{\S}$ has two vertices with weight 0 and $n+1$ edges joining both of them. (Just take $k=n$ in the above theorem.)
 \item If $n=6$, there are multicurves $\S$ such that $\G_{\S}$ has two vertices with the same weight and either 3,5 or 7 edges.
 \end{enumerate}

  \subsection{Multicurves $\S$ consisting in one closed curve.}
  \begin{teo}\label{thm:Piramides1curva}
  \begin{itemize}
  \item[(a)] Suppose $\S=\{\g\}$ with $\g$ a simple closed curve. 
 Then, there is a number $m\geq 1$ dividing $n$ so that the graph $\G_{\S}$ has $\frac{n}{m}+1$ vertices, one with weight 0 and degree $n$ and the others with weight 1 and degree $m$, and $n$ edges, each of them  joins the vertex of weight 0 with a different vertex (see Figure \ref{fig:graphs}(3)). 
 \item[(b)] Conversely, for any $m\geq 1$ dividing $n$, there is multicurve $\S=\{\g\}$ with $\g$ a closed curve, such that its associated graph $\G_{\S}$ is the graph described above. 
  \end{itemize}

  \end{teo}
  \begin{proof} (a)  The curve $\g$ separates $\O$ into two discs, i.e., $\O\menos \g=\O_1\cup \O_2$, where   $\O_1$   contains  the cone point of order $n$ and a cone point of order 2 and $\O_2$ contains the remaining  three cone points of order 2. We choose basepoints and paths as explained in Section \ref{sec:ChoicePaths}, i.e., choose $*=*_1\in \O_1$, $*_2\in \O_2$ and $*_{\g}\in \g$,  we take  $\b_1$ to be the constant path, we choose $\b_{1,\g}$ and $\b_{2,\g}$, and we  take $\b_{\g}=\b_{1,\g}$ and $\b_2=\b_{1,\g}\b_{2,\g}^{-1}$.

  Since $\O_1$ contains the cone point of order $n$,  $\Im\Phi_1$  is generated by an element of order $n$ and by a symmetry of $D_n$. Thus, $\Im\Phi_1=D_n$ and $|\Im\Phi_1|=2n$. 
On the other hand,  $\Im\Phi_2$ is generated by three symmetries of $D_n$, and therefore is isomorphic to a dihedral group, say $D_m$ for some $m$ dividing $n$, so that $|\Im\Phi_2|=2m$. 
Then, by Theorem\ref{thm:Combinatotial}(i), $\G_{\S}$ has  $\frac{2n}{2n}+\frac{2n}{2m} =1+\frac{n}{m}$ vertices. 

To compute the number of edges, we see that $\Im\Phi_{\g}$ is generated by the product of an element of order $n$ and a symmetry, and this product is always a symmetry. Thus $|\Im\Phi_{\g}|=2$ and, by Theorem \ref{thm:Combinatotial}(ii), $\G_{\S}$ has $\frac{2n}{2}=n$ edges. 

Using  parts (iii) and (iv) of that theorem, we can easily compute the degree and weight of the vertices. Indeed, the vertex $V_{1,e}$ has degree $D_1= 2n\frac{1}{2}=n$ and weight $w^1=1- \frac{1}{2}(2n(\frac{1}{2}-\frac{n-1}{n})+n)= 0$. Notice that all the other vertices $V_{2,g},g\in D_n$ have the same degree and weight, $D_2=2m\frac{1}{2}=m$ and $w^2=1- \frac{1}{2}(2m(-\frac{1}{2})+m)=1$.

Finally, we use Theorem \ref{thm:Combinatotial}(v) to see how the  edges connect vertices. 
Notice that both  loops $\b_{\g}\b_{1,\g}^{-1}\b_1^{-1}$ and $\b_{\g}\b_{2,\g}^{-1}\b_2^{-1}$ are homotopically trivial, thus, 
the edge $E_{\g,g\Im \Phi_{\g}}$ joins the vertex $V_{1,g\Im\Phi_1} $ to the vertex $V_{2,g\Im\Phi_2}$. Since $\Im\Phi_1=D_n$, all the vertices $V_{1,g\Im\Phi_1}$ are equal to $V_{1,e}$, and since 
$\Im\Phi_{\g}$ is conjugate to a subgroup of   $\Im\Phi_2$, then there are $\frac{|\Im\Phi_2|}{|\Im\Phi_{\g}|}=\frac{2m}{2}$ edges coming into any vertex $V_2,g$.  This proves the result.

   
 \psfrag{t}{$t$} 
  \psfrag{t=0}{$t=0$}
   \psfrag{t=1}{$t=1$}
 \begin{figure}
\center 
\includegraphics[height=10cm,width=14cm]{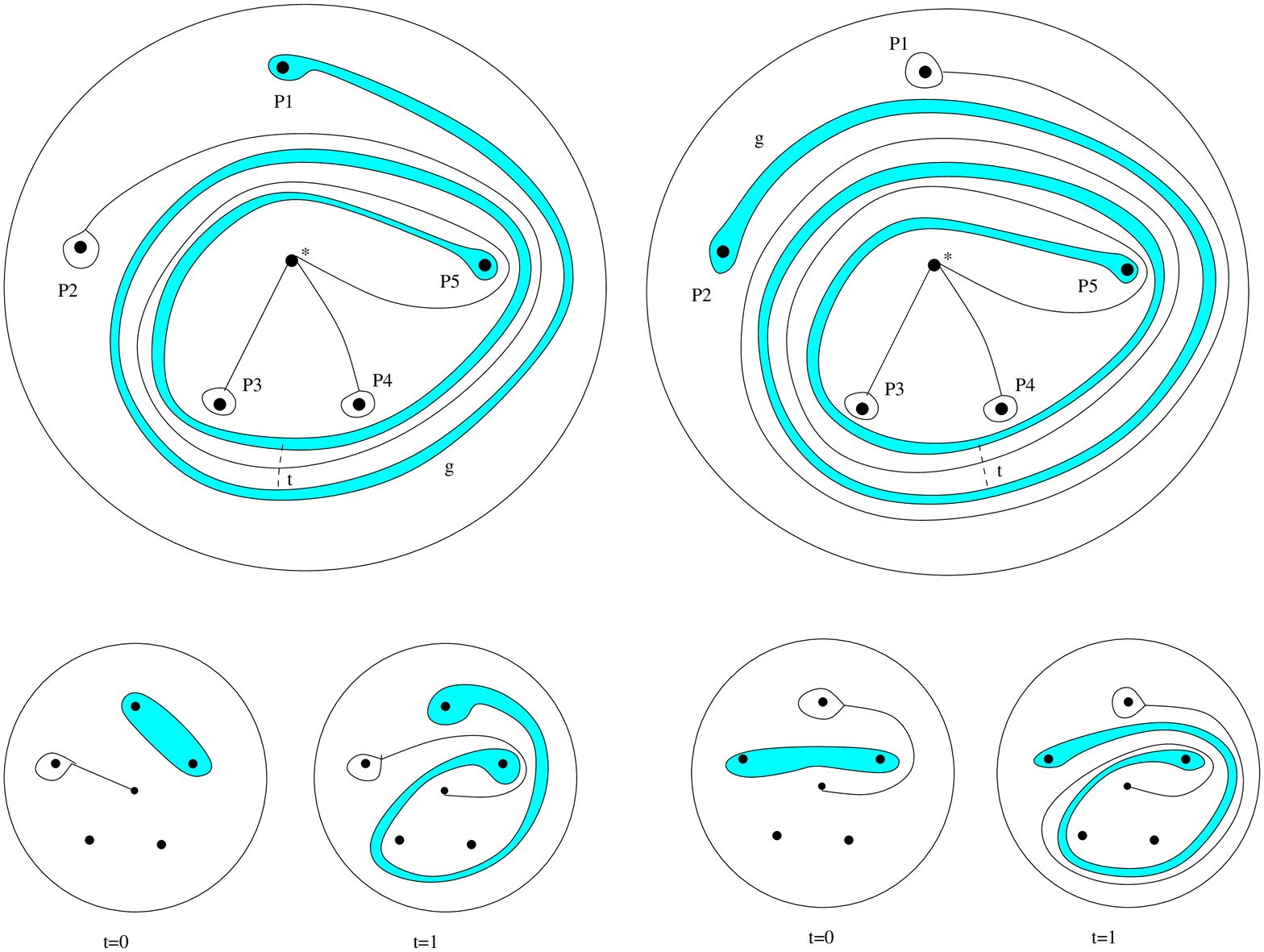}
\caption{}
\label{fig:Piramides1curva}
 \end{figure}
   
(b)   Given $m\geq 1$ dividing $n$, we need to find a closed curve $\g$ so that, with the notations of the first part of the proof, $\Im\Phi_2$ be isomorphic to the dihedral group $D_m$. Take the curve $\g$ as shown in Figure \ref{fig:Piramides1curva}. We think of  $\g$ as bounding  a disc neighborhood of an arc that joins  $P_5$ to another cone point and  wrapping $t\geq 0$ times around another two cone points. 

On the left  part of the figure,  $\Im\Phi_2$ is generated by the three words $\rho\sigma, \rho\sigma$ and  $ \rho^t(\rho\sigma)\rho^{-t} = \rho^{2t+1}\sigma$, i.e., 
$$
\Im\Phi_2=\langle \rho\sigma , \rho^{2t }\rangle.
$$
On the right part of the figure, $\Im\Phi_2$ is generated by the three words $\rho\sigma, \rho\sigma$ and  $ \rho^{t+1}( \sigma)\rho^{-t-1} = \rho^{2t+2 }\sigma$, i.e., 
$$
\Im\Phi_2=\langle \rho\sigma , \rho^{2t+1}\rangle.
$$
The cases $t=0,t=1$ are shown on the bottom of the figure. In conclusion, by taking different values of $t$ we obtain all possible dihedral subgroups of $D_n$, so   the result is proved.
Explicitly, suppose $\frac{n}{m}$ is even. Since the order of $\rho^{2t }$ is $\frac{n}{(n,2t)}$, then this order is equal to $m$ if and only if $(n,2t)=\frac{n}{m}$, and we can take  $t=\frac{n}{2m}$ and use the figure on the left. 
Similarly, if $\frac{n}{m}$ is odd, then we can take 
   $t=\frac{1}{2}\left(\frac{n}{m}-1\right)$   and use the figure on the right.
\end{proof}


  \psfrag{m=1}{$m=1$}
   \psfrag{m=2}{$m=2$}
   \psfrag{m=3}{$m=3$}
   \psfrag{m=6}{$m=6$}
   \psfrag{0}{$0$}
   \psfrag{1}{$1$}
 
 \subsection{Multicurves $\S$ consisting in an arc and a closed curve.}
 \begin{teo}\label{thm:Piramides1arc1curve} 
\begin{itemize}
\item[(a)] Let $\S=\{\g_1,\g_2\}$, where $\g_1$ is  an arc  and $\g_2$ is a  closed curve. Then there exists $m\geq 1$ dividing $n$ so that the graph $\G_{\S}$ has:
 \begin{itemize}
 \item[i)] $\frac{n}{m}$ vertices $V_0 ,V_1,\dots, V_{n/m}$ all of them with weight 0, $V_0$   with degree $n$ and the others with degree $m+2$;
 \item[ii)] $n+\frac{n}{m}$ edges;
 \item[iii)] for each vertex $V_i, i>0$  there are $m$ edges joining it to $V_0$;
 \item[iv)] the remaining edges join the vertices $V_i, i>0$ in cycles of length $d$ for some $d$ dividing  $  n/m$  (see Figure \ref{fig:graphs}(4)).
\end{itemize}  
\item[(b)] For any $m $ dividing $n$ and:
\begin{itemize}
\item[i)]  $d=\frac{n}{m}$, or
\item[ii)] $d=1$,  or
\item[iii)] $n/m$ is even and  $d=2$
\end{itemize}
  there is a multicurve $\S$ such that $\G_{\S}$ is the graph described in part (a).  
\end{itemize}

 \end{teo}
  \begin{proof}   
 (a)  In this situation, we have $\O\menos\S=\O_1\cup\O_2$, where $\O_1$ is an annulus with a cone point of order 2, and $\O_2$ is a disc with a cone point of order 2 and a cone point of order $n$. Then we have: 
\begin{itemize}
\item  $\Im\Phi_2$ is the whole group, by the same argument as in Theorem \ref{thm:Piramides1curva}. Thus,  $|\Im\Phi_2|=2n$.
\item   Also as in Theorem \ref{thm:Piramides1curva}, $\Im\Phi_{\g_2}$ is generated by a symmetry, so that $|\Im\Phi_{\g_2}|=2$.
\item   $\Im\Phi_{\g_1}=\langle \Phi_{\g_1}(\g_1^a), \Phi_{\g_1}(\g_1^a\g_1^b) \rangle$ is, by Lemma \ref{lem:ImPhi_arco}, isomorphic to $D_m$ for some $m\geq 1$ dividing $n$. Thus, $|\Im\Phi_{\g_1}|=2m$.  Notice that $\b_{\g_1}(\b_{1,\g_1}^b)^{-1}\b_1^{-1}$ is homotopic to either $\b_{\g_1}\g_1^a\b_{\g_1}^{-1}$ or $\b_{\g_1}\g_1^b\b_{\g_1}^{-1}$, so that 
 $\Phi(\b_{\g_1}(\b_{1,\g_1}^b)^{-1}\b_1^{-1})\in \Im\Phi_{\g_1}$   is a symmetry, that we denote  by $S$.
\item    Finally, $\pi_1(\O_1,*_1)$ is generated by the curve $\g'_1=\b_{1,\g_1}\g_1^a\g_1^b(\b_{1,\g_1})^{-1}$, which  surrounds the arc $\g_1$,
 and by  a loop $z$ surrounding the cone point of $\O_1$.   Notice that $\Phi_1(\g_1')=\Phi_{\g_1}(\g_1^a\g_1^b)$ and that  $\Phi_1(z)=\Phi(z)$ is a symmetry, that we denote by $s$.
Then,  $\Im\Phi_{ 1}=\langle  \Phi_{\g_1}(\g_1^a\g_1^b), s\rangle$ 
  is also isomorphic to $D_m$. Notice that $\Im\Phi_{\g_1}$ and $\Im\Phi_1$ are not necessarily equal, but share the subgroup $C_m$  of index 2 containing all their  rotations.  
 
\end{itemize} 
 
 \noindent  For later purposes, we define $R=Ss$, and let $d$ be the minimum number such that $R^d\in \Im\Phi_1$. If we call $C_n$ to the subgroup of rotations of $D_n$ and $C_m$ to the subgroup of rotations of $\Im\Phi_1$, we have that $d$ is the order of the coset $R\cdot C_m$ in $C_n/C_m$. In particular, $d$ divides $n/m$.
 
 \medskip
 
 With the above information and using Theorem \ref{thm:Combinatotial}, we have that $\G_{\S}$ has $1+\frac{n}{m}$ vertices, one of them, $V_{2,e}$, projecting over $\O_2$, and the remaining $V_{1,g}$ projecting over $\O_1$. The graph $\G_{\S}$ has   $n+\frac{n}{m}$ edges, $n$ of them $E_{\g_2,g}$ projecting over $\g_2$ and the remaining $E_{\g_1,g}$ projecting over $\g_1$. The degree of the vertex $V_{2,e}$ is equal to $D_2= 2n\frac{1}{2}=n$, and its weight is $w^2=1-\frac{1}{2}(2n(\frac{1}{2}-\frac{n-1}{n})+n)=0$. The degree of the vertices $V_{1,g}$ is equal to $D_1=2m(\frac{1}{2}+2\frac{1}{2m})=m+2$, and their weight is $w^1=1-\frac{1}{2}(2m(-\frac{1}{2})+m+2)=0$.

  \psfrag{D}{$\Delta$}
  \psfrag{*}{$*$}
    \psfrag{*2}{$*_2$}
    \psfrag{*g1}{$*_{\g_1}$}
    \psfrag{*g2}{$*_{\g_2}$}
  \psfrag{O1}{$\O_1$}
   \psfrag{O2}{$\O_2$}
   \psfrag{2}{$2$}
    \psfrag{n}{$n$}
   \psfrag{b2g1}{$\b_{2,\g_1}$}
   \psfrag{b2g2}{$\b_{2,\g_2}$}
   \psfrag{g1}{$ \g_1$}
   \psfrag{g2}{$ \g_2$}
   \psfrag{t}{$ t$}

 
 It is left to show how the edges connect the vertices, for which we will use Theorem \ref{thm:Combinatotial}(v). We choose the basepoints and paths as explained in Section \ref{sec:ChoicePaths}.
 Using the same argument as in the proof of Theorem \ref{thm:Piramides1curva}, we obtain that for each vertex $V_{1,g}$ there are $m$ edges projecting over $\g_2$ joining it with    $V_{2,e}$. 
 
 On the other hand, since  $\g_1$    is in the closure of just  $\O_1$, the   edges $E_{\g_1,g}$ join vertices $V_{1,g}$ among them. 
   The edge $E_{\g_1,e}$ joins the vertices $V_{1,e}$ and $V_{1,S}$, and this last vertex is equal to $V_{1,R}$ because   $S\Im\Phi_1=Ss\Im\Phi_1$. Then we have that: the edge $E_{\g_1, R}$ joins the vertices $V_{1,R}$ and $V_{1,R^2}$, the edge $E_{\g_1, R^2}$ joins the vertices $V_{1,R^2}$ and $V_{1,R^3}$, etc. Clearly, $V_{1,R^d}=V_{1,e}$. We claim that all the  vertices $V_{1,R^k}$, for $k=1, \dots, d$ are different. Indeed, let $k,k'=1,\dots, d$. If $V_{1,R^k}=V_{1,R^{k'}}$ then $R^{k-k'}\in \Im\Phi_1$, which contradicts the definition of $d$. 
   This shows that there is a cycle of  $d$  vertices $V_{1,R^k}$ joined by edges.  Moreover, each vertex $V_{1,R^k}$  is joined to $V_{2,e}$  with $m$ edges and it is also joined to $V_{1,R^{k-1}}$, $V_{1,R^{k+1}}$. Since we already know that its degree is $m+2$, then this vertex is not joined to any other vertex. 
 \hfill  \break
    Now, suppose $V_{1, g}, g\in D_n$ is a vertex different from all the previous ones. Then the edge $E_{\g_1, gR^k}$ joins the vertices $V_{1,gR^k}$ and $V_{1,gR^{k+1}}$ for all $k=1,\dots, d$. 
Clearly, these  vertices are all different, because $V_{1,R^k}$ are. They are also different from the vertices $V_{1,R^k}$, otherwise some vertex would have degree greater than $m+2$. 
\hfill \break
In conclusion,  there are $(n/m)/d$ cycles  of length $d$ of  vertices projecting over $\O_1$. So the result is proved.  
 
\medskip
 \begin{figure}
\center 
\includegraphics[height=18cm,width=14cm]{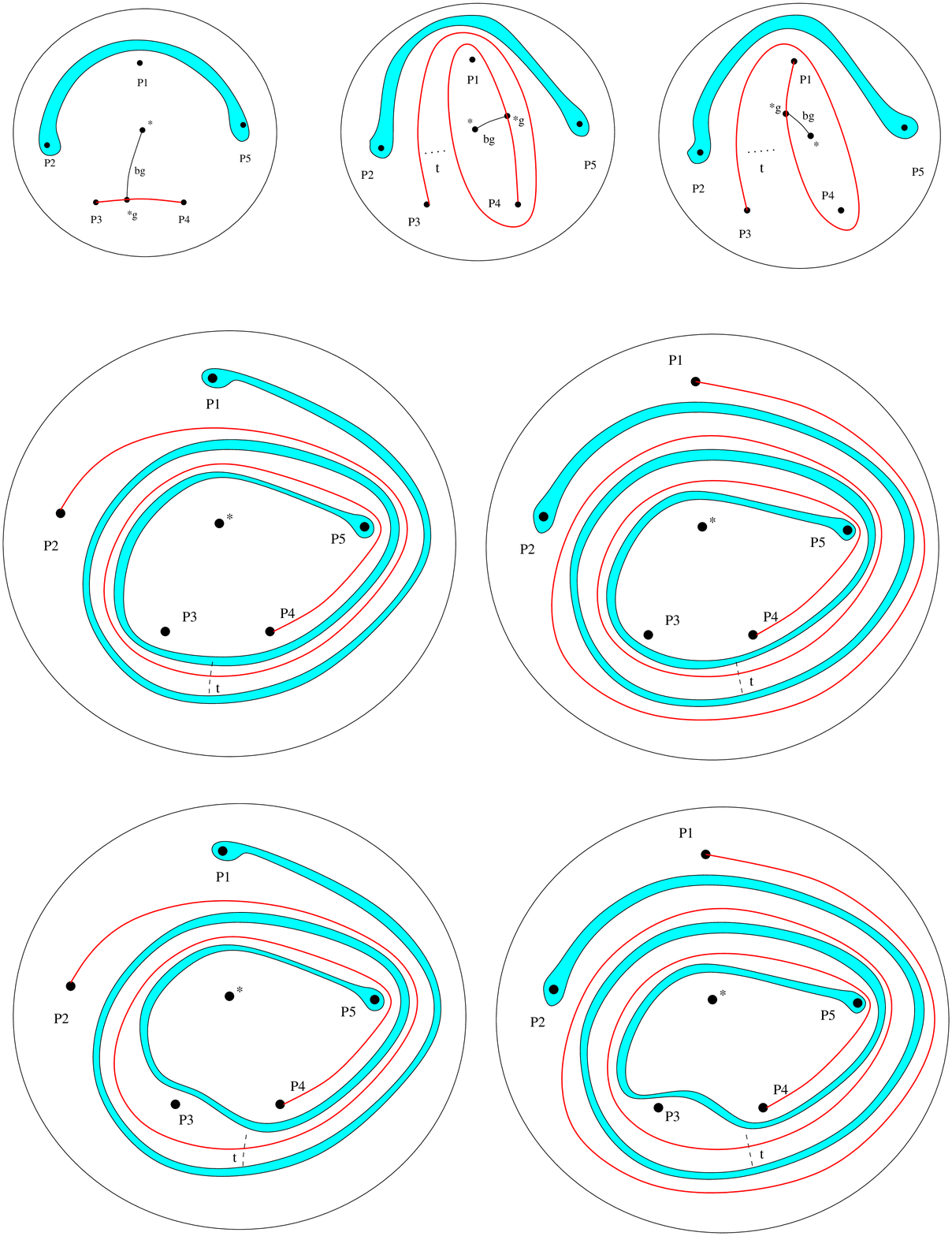}
\caption{Examples for Theorem \ref{thm:Piramides1arc1curve}}
\label{fig:Piramides1curva1arco-II}
 \end{figure}

(b)  To prove this part, given $m$ and $d$, we will find appropriate examples of multicurves $\S=\{\g_1,\g_2\}$  such that $\Im\Phi_{\g_1}$ be isomorphic to $D_m$ and, with the notations introduced in the proof of part (a), so that  $d$ be order of the coset $R\cdot C_m$ in $C_n/C_m$.

i) We consider first  the curves $\g_1,\g_2$ on the top of  Figure \ref{fig:Piramides1curva1arco-II}, where $\g_1$ is   the curve in  one of the three cases of Figure \ref{fig:Piramides1arco}. In Lemma \ref{lem:ImPhi_arco} we showed that $\g_1$ can be chosen so that $\Im\Phi_{\g_1}$ is isomorphic to $D_m$ for any $m$ dividing $n$. Now, just notice that, in the three figures, we can take $S=\rho\sigma$ and $s=\sigma$ or $S=\sigma$, $s=\rho\sigma$.  Thus, $R=Ss=\rho^{\pm 1}$ and the order of $R\cdot C_m$ in   $C_n/C_m$ is  $d=\frac{n/m}{(\frac{n}{m},1)}=\frac{n}{m}$.  Hence,   all the  vertices $V_{1,g}$ are connected by edges forming one cycle.

 \medskip

 ii) Next we consider  the curves $\g_1,\g_2$ on the middle of  Figure \ref{fig:Piramides1curva1arco-II}, where $\g_2$ is the curve $\g$ in  the examples in the proof of Theorem \ref{thm:Piramides1curva}(b).
On the left, we  have $S=\Phi( {\g_1}^a)= \rho\sigma$,  $\Phi( {\g_1}^b)=\rho^t\rho\sigma\rho^{-t}$ and $s=\Phi(z)=\rho\sigma$. So we have:
$$
\Im\Phi_{\g_1}=\langle  \rho\sigma, \rho^t(\rho\sigma)\rho^{-t} \rangle 
=\langle  \rho\sigma,  \rho^{2t} \rangle 
,\quad
\Im\Phi_1=\langle \rho\sigma, \rho^{2t}
 \rangle
$$
Similarly, on the right part of the figure, we have $S=\Phi( {\g_1}^a)= \rho\sigma$,  $\Phi( {\g_1}^b)=\rho^{t+1}\sigma\rho^{-t-1}$ and $s=\Phi(z)=\rho\sigma$.
$$
\Im\Phi_{\g_1}=\langle  \rho\sigma, \rho^{t+1} \sigma \rho^{-t-1} \rangle 
=\langle  \rho\sigma,  \rho^{2t+1} \rangle 
,\quad
\Im\Phi_1=\langle \rho\sigma, \rho^{2t+1}
 \rangle.
$$
Thus, in  both cases we have  $R=e$ and $d=1$.   Since, as varying $t$,  $\Im\Phi_{\g_1}$ is any possible dihedral subgroups of $D_n$, we have that the graph $\G_{\S}$ has $\frac{n}{m}$ vertices with one loop each.

\medskip

iii) Finally we consider  the curves on the bottom of Figure \ref{fig:Piramides1curva1arco-II}. In this case,
$\Im\Phi_{\g_1}$ is as in case i), but now $s= \rho(\rho\sigma)\rho^{-1}=\rho^{3}\sigma$. Thus, $R=Ss= (\rho\sigma)( \rho^3\sigma)=\rho^{-2}$. On the other hand, we have that 
$\Im\Phi_1=\langle  \rho^3\sigma, \rho^{2t}
 \rangle$ for the left-sided figure and $\Im\Phi_1=\langle \rho^3\sigma, \rho^{2t+1}\rangle$ for the left-sided figure.
Thus, varying $t$ we obtain that $\Im\Phi_1 \cong D_m$ for all $m$ dividing $n$. 
Then, the order $d$  of $R$ in   $C_n/C_m$ is 
$$ d=\frac{\frac{n}{m}}{(\frac{n}{m},2)}
=
\left\{ 
\begin{matrix}
n/m & {\rm if} & n/m & \hbox{is odd}
\\
\frac{1}{2}(n/m) & {\rm if} & n/m & \hbox{is even}.
\end{matrix} 
\right.
$$
Hence, if $n/m$ is odd, the vertices $V_{1,g}$ are as in the case ii). While, if $n/m$ is even, the vertices $V_{1,g}$ are joined by edges forming two cycles of length $\frac{1}{2}(n/m)$.
\end{proof}

 \bigskip
 
 
 The previous theorems permit to describe exactly which strata of $\partial \widehat{\M_n}$ are in the limit of $\mathcal P_n$ when $n$ is prime and when $n=4$, whose only nontrivial divisor is 2.
 
\begin{cor}
\label{cor:n=CasiPrimo}
If $n=4$ or if  $n=p $   with $p$ a prime number, then the strata of $\partial \widehat{\mathcal{P}_n}\cap \widehat{\M_n}$ are exactly those described in Theorems \ref{thm:Piramides1arco}, \ref{thm:Piramides2arcos},  \ref{thm:Piramides1curva}, \ref{thm:Piramides1arc1curve}. See Figure \ref{fig:Grafos-n3n4} for the stable graphs describing these strata for the cases $n=3$, $n=4$. 
\end{cor} 

\psfrag{1}{1}
\psfrag{2}{2}
\psfrag{0}{0}
\psfrag{3}{3}
\psfrag{n=3}{$n=3$}
\psfrag{n=4}{$n=4$}
\psfrag{1arc}{1 arc}
\psfrag{2arcs}{2 arcs}
\psfrag{1arc1curve}{1 arc, 1 c. curve}
\psfrag{1curve}{1 c. curve}
\begin{figure}
\includegraphics[height=8cm,width=15cm]{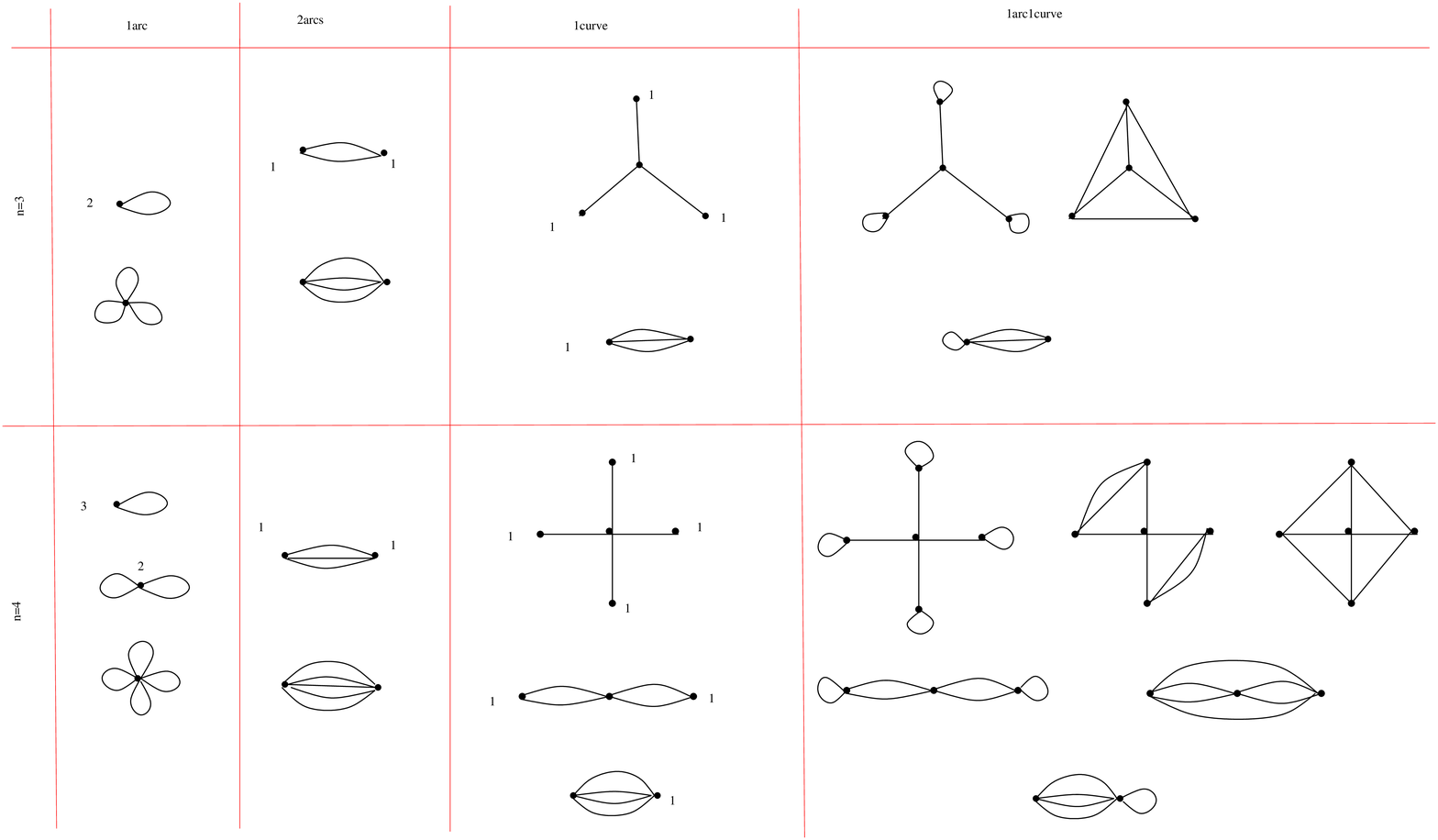} 
 \label{fig:Grafos-n3n4}
 \caption{Stable graphs corresponding to the strata in the boundary of $\mathcal{P}_n$ for $n=3,4$. No number assigned to a vertex means weight equal zero.}
\end{figure}



\begin{thebibliography}{99} 

\bibitem {abikoff}  Abikoff, W. \textbf{Degenerating families of Riemann surfaces.}
Annals of Math. Vol 105.  no. 1.    29-94, 1977.




\bibitem {archer}  Archer, J.D.and  Pries, R. \textbf{The integral monodromy of hyperelliptic and trielliptic curves.}
Math. Ann.  338,  187-206, 2007.

\bibitem {ashikaga}  Ashikaga T., and  Ishizaka, M. \textbf{Classification of degenerations of curves of genus three via Matsumoto-Montesinos theorem.} 
Tohoku Math J.  54, 195-226, 2002.

                                                                                              
                                                                                                                                                                                                                                                                      \bibitem {bartolini}  Bartolini G., Costa A. F. and  Izquierdo, M. \textbf{On the connectivity of the branch loci of moduli spaces.}
Ann. Acad. Sci. Fenn.  Math.  38  , no. 1, 245-258, 195-226, 2013.

\bibitem {bers} Bers, L. \textbf{On spaces of Riemann surfaces with nodes.}
Bulletin of the AMS. 80, Number 6, 1219-1222, 1974.

\bibitem {porti}  Boileau, M., Maillot, S and Porti, J. \textbf{Three-Dimensional Orbifolds and their Geometric Structures.}  Panoramas et Synth\`eses, no. 15.  Soci\'et\'e Math\'ematique
de France, 2003.

 \bibitem {Broughton}  Broughton, S. L. \textbf{The equisymmetric stratification of the moduli space and the Krull dimension of mapping class groups.}
Topology and its applications. 37, 101-113, 1990.

\bibitem {Broughton-Classifying}  Broughton, S. L. \textbf{Classifying finite group actions on surfaces of low genus.}
J. of Pure and Applied Algebra. 69, 233-270, 1990.

\bibitem {cornalba}  Cornalba, M. \textbf{On the locus of curves with automorphisms.}  Annali di Matematica Pura ed Applicata.
149, no. 1, 135--151, 1987.

\bibitem {costa0}  Costa, A.F. and  Izquierdo, M.  \textbf{On the connectedness of the branch locus of the moduli space of Riemann surfaces of genus $g=4$.}  Glasg.
Math. J.   52,  401--408, 2010.



\bibitem {costa1}  Costa, A.F., Izquierdo, M. and Parlier, H. \textbf{Connecting the p-gonal loci in the compactification of the moduli space.}  Rev Mat Complut.
28, no. 1, 469--486, 2015.

\bibitem{costa2}  Costa, A. F. and Gonz\'alez-Aguilera, V. \textbf{Limits of equisymmetric 1-complex dimensional families of Riemann surfaces}, to appear in Mathematica Scandinavica, 2017.

\bibitem {deligne} Deligne, P. and Mumford, D. \textbf{The irreducibility of
the space of curves.} Publications Math\'{e}matiques de L'I.H.E.S. 36 ,
75-109, 1965.

\bibitem {victor} Gonz\'{a}lez-Aguilera, V. and Rodr\'{\i}guez, R.E.
\textbf{On principally polarized abelian varieties and Riemann surfaces associated to prisms and pyramids.} Contemporary Mathematics. 211, AMS, 269-284, 1997.


\bibitem{harvey} Harvey, W. \textbf{Chabauty spaces of discrete groups.} In L. Greenberg, editor: Discontinuous groups and Riemann Surfaces, Vol. 79 Princeton University Press, 295-348, 1974.


\bibitem {Hubbard}  Hubbard, J. H.  and Koch, S. \textbf{An analytic construction of the Deligne-Mumford compactification of the moduli space of curves.} J. Differential Geometry, 98, 261-313, 2014. 

\bibitem {kimura}  Kimura, H. \textbf{Classification of automorphism groups, up to topological equivalence, of compact Riemann surfaces of genus 4.}  J. of Algebra, 264, 26-54, 2003.




\bibitem {miranda} Miranda, R. \textbf{Graph curves and curves on K3 surfaces.}
International Centre for Theoretical Physics, Trieste. World Scientific.
119-176, 1989.

\bibitem {montesinos} Matsumoto, Y. and Montesinos-Amilibia, J. M.
\textbf{Pseudo-periodic maps and degenerations of Riemann surfaces.} Lectures
Notes in Mathematics volume 2030. Springer-Verlag, 2011.


\end{thebibliography}
\end{document}